\documentclass[12pt]{amsart}
\usepackage{amsmath,amsthm,amsfonts,amscd,amssymb,eucal,latexsym,mathrsfs,enumerate,framed,tocvsec2}
\usepackage[all]{xy}
\CompileMatrices % quicker compilation of c-diags

\usepackage{hyperref}
\numberwithin{equation}{section}

\theoremstyle{plain}
\newtheorem{theorem}{Theorem}
\newtheorem{lemma}[theorem]{Lemma}
\newtheorem{corollary}[theorem]{Corollary}
\newtheorem{proposition}[theorem]{Proposition}
\newtheorem{properties}[theorem]{Properties}

\theoremstyle{definition}

\newtheorem{question}[theorem]{Question}
\newtheorem{definition}[theorem]{Definition}
\newtheorem{remark}[theorem]{Remark}

\newcommand{\Hs}{\mathcal H}
\newcommand{\Ks}{\mathcal K}
\newcommand{\Ad}{\mathrm{Ad}}
\newcommand{\ad}{\mathrm{ad}}
\newcommand{\id}{\mathrm{id}}
\newcommand{\cb}{\mathit{cb}}

\newcommand{\N}{\mathcal N} % Used for normaliser
\newcommand{\U}{\mathcal U} % Used for unitary
\renewcommand{\sp}{\mathrm{Sp}} % Used for spectra
\newcommand{\eps}{\varepsilon}

\newcommand{\Z}{\mathcal Z}
\newcommand{\vnotimes}{\,\overline{\otimes}\,}
\newcommand{\row}{\mathrm{row}}
\newcommand{\alt}{\mathrm{alt}}

\begin{document}

\title{Kadison--Kastler stable
factors}

\author[Cameron et. al.]{Jan Cameron}
\address{\hskip-\parindent
Jan Cameron, Department of Mathematics, Vassar College, Poughkeepsie, NY 12604, 
U.S.A.}
\thanks{JC's research is partially supported by an AMS-Simons research travel grant.}
\email{jacameron@vassar.edu}

\author[]{Erik Christensen}
\address{\hskip-\parindent
Erik Christensen, Institute for Mathematiske Fag, University of Copenhagen,
Copenhagen, Denmark.}
\email{echris@math.ku.dk}
\author[]{Allan M.~Sinclair}
\address{\hskip-\parindent
Allan M.~Sinclair, School of Mathematics, University of Edinburgh, JCMB, King's
Buildings, Mayfield Road, Edinburgh, EH9 3JZ, Scotland.}
\email{a.sinclair@ed.ac.uk}
\author[]{Roger R.~Smith}
\address{\hskip-\parindent
Roger R.~Smith, Department of Mathematics, Texas A{\&}M University,
College Station, TX 77843,  U.S.A.}
\email{rsmith@math.tamu.edu}
\thanks{RS's research is partially supported by NSF grant DMS-1101403}

\author[]{Stuart White}
\address{\hskip-\parindent
Stuart White, School of Mathematics and Statistics, University of Glasgow, 
University Gardens, Glasgow Q12 8QW, Scotland.}
\email{stuart.white@glasgow.ac.uk}
\thanks{SW's research is partially supported by EPSRC grant EP/I019227/1.}

\author[]{Alan D.~Wiggins}
\address{\hskip-\parindent
Alan D.~Wiggins, Department of Mathematics and Statistics, University of
Michigan-Dearborn,  Dearborn, MI 48126, U.S.A.}
\email{adwiggin@umd.umich.edu}

\maketitle

\begin{abstract}
A conjecture of Kadison and Kastler from 1972 asks whether sufficiently close
operator algebras in a natural uniform sense must be small unitary perturbations
of one another. For $n\geq 3$ and a free ergodic probability measure preserving
action of $SL_n(\mathbb Z)$ on a standard nonatomic probability space $(X,\mu)$,
write $M=((L^\infty(X,\mu)\rtimes
SL_n(\mathbb Z))\vnotimes R$, where $R$ is the hyperfinite II$_1$ factor. We
show that whenever $M$ is represented as a von
Neumann algebra on some Hilbert space $\Hs$ and $N\subseteq\mathcal B(\Hs)$ is
sufficiently close to $M$, then there is a unitary $u$ on $\Hs$ close to the
identity operator with $uMu^*=N$.
 This provides the first nonamenable class of von Neumann algebras satisfying
Kadison and Kastler's conjecture. 
 
We also obtain stability results for crossed products
$L^\infty(X,\mu)\rtimes\Gamma$ whenever the comparison map from the bounded to usual
group cohomology vanishes in degree $2$ for the module $L^2(X,\mu)$.  In this
case, any von Neumann algebra sufficiently close to such a crossed product is
necessarily isomorphic to it. In particular, this result applies when $\Gamma$
is a free group.

This paper provides a complete account of the results announced in \cite{CCSSWW:PNAS}.
\end{abstract}

\section{Introduction}

\renewcommand*{\thetheorem}{\Alph{theorem}}

In \cite{KK:AJM}, Kadison and Kastler introduced a metric $d$ on the collection
of all closed subalgebras of the bounded operators on a Hilbert space in terms of
the Hausdorff distance between the unit balls of two algebras $M$ and $N$, and
conjectured that sufficiently close operator algebras should be isomorphic. 
Qualitatively, $M$ and $N$ are close in the Kadison-Kastler metric if each
operator in the unit ball of $M$ is close to an operator in the unit ball of $N$
and vice versa.  Canonical examples of close operator algebras are obtained by
small unitary perturbations: given an operator algebra $M$ on a Hilbert space
$\Hs$ and  a unitary operator $u$ on $\Hs$ close to the identity operator, then
$uMu^*$ is close to $M$. The strongest form of the Kadison-Kastler conjecture
states that every algebra sufficiently close to a von Neumann algebra $M$ arises in this fashion.  This has  been established when $M$ is an
injective von Neumann algebra \cite{C:Invent,RT:JFA,Joh:PLMS,C:Acta} (building
on the earlier special cases in \cite{C:JLMS,Ph:PJM}) but remains open for
general von Neumann algebras.  

We now present the central result of the paper: Theorem \ref{TA}. This has been announced in our short survey article \cite{CCSSWW:PNAS} which contains a heuristic discussion of our methods but no formal proofs.
\begin{theorem}\label{TA}
Let $n\geq 3$ and let $\alpha:SL_n(\mathbb Z)\curvearrowright(X,\mu)$ be a free,
ergodic and measure preserving action of $SL_n(\mathbb Z)$ on a standard nonatomic
probability space $(X,\mu)$.  Write $M=(L^\infty(X,\mu)\rtimes_\alpha
SL_n(\mathbb Z))\vnotimes R$, where $R$ is the hyperfinite II$_1$ factor.  For 
 $\eps>0$, there exists $\delta>0$ with the following property:  given
a normal unital representation $M\subseteq\mathcal B(\Hs)$ and another von
Neumann algebra $N$ on $\Hs$ with $d(M,N)<\delta$, there exists a unitary
$u\in\Hs$ with $\|u-I_\Hs\|<\eps$ and $uMu^*=N$.
\end{theorem}

Theorem \ref{TA} provides the first nonamenable II$_1$ factors
which satisfy the strongest form of the Kadison-Kastler conjecture. A key ingredient in this result is the vanishing of
the bounded cohomology groups $H^2_b(SL_n(\mathbb Z),L^\infty_\mathbb R(X,\mu))$
for $n\geq 3$ from \cite{M:Crelle,BM:GAFA,MS:JDG} and in Theorem \ref{TA}, which is then an immediate consequence of Theorem \ref{SSKK}, $SL_n(\mathbb Z)$ can be replaced with any other group with this property.  Via
the work of \cite{Bo:JAMS,P:Invent2,P:Invent3,P:IMRN}, there are uncountably many pairwise nonisomorphic II$_1$
factors to which this theorem applies (see Remark \ref{uncountable}).  

The Kadison-Kastler conjecture is known to be false in full generality. In \cite{CC:BLMS},
examples of arbitrarily close nonseparable and nonisomorphic $C^*$-algebras were
found, while in \cite{Joh:CMB} Johnson presented examples of arbitrarily close
unitarily conjugate pairs of separable nuclear $C^*$-algebras where the
implementing unitaries could not be chosen to be close to the identity operator.
Thus the appropriate form of the conjecture for $C^*$-algebras is that
sufficiently close separable $C^*$-algebras should be isomorphic or spatially
isomorphic. In this last form, the conjecture has been settled affirmatively for
close separable nuclear $C^*$-algebras on separable Hilbert spaces
\cite{CSSWW:Acta} (see also \cite{CSSWW:PNAS}) with earlier special cases
established in \cite{C:Acta,PhR:CJM,PhR:PLMS,Kh:MMJ}. Our methods also give
examples of nonamenable von Neumann algebras satisfying these weaker forms of
the conjecture, as we now state. The hypotheses on 
 the action in the following theorem ensure that $M$ is a  II$_1$ factor with separable predual
satisfying $P'\cap
M\subseteq P$. The three parts of Theorem \ref{TB} are proved in Section \ref{KKStable} as Corollary \ref{MainCor}, Corollary \ref{spatialKK} and Theorem \ref{StrongKK} respectively.

\begin{theorem}\label{TB}
Let $\alpha:\Gamma\curvearrowright P$ be a centrally ergodic, properly outer and
trace-preserving action  of a countable discrete group $\Gamma$ on
a finite amenable von Neumann algebra $P$ with separable predual and write
$M=P\rtimes_\alpha \Gamma$.  
\begin{enumerate}
\item\label{TB1} Suppose that the comparison map 
\begin{equation}\label{TBCM}
H^2_b(\Gamma,L^2(\Z(P)_{sa}))\rightarrow H^2(\Gamma,L^2(\Z(P)_{sa}))
\end{equation}
from bounded cohomology to usual cohomology vanishes, where $\Z(P)$ denotes the
center of $P$. Then, given a normal unital representation $M\subseteq \mathcal
B(\Hs)$, each von Neumann algebra $N$ on $\mathcal B(\Hs)$ sufficiently close to
$M$ is isomorphic to $M$.
\item\label{TB2} Suppose that the comparison map (\ref{TBCM}) vanishes and that
$M$ has  property $\Gamma$. Then, given a normal unital representation
$M\subseteq \mathcal B(\Hs)$, each von Neumann algebra $N$ on $\mathcal B(\Hs)$
sufficiently close to $M$ is spatially isomorphic to $M$.
\item\label{TB3} Suppose that the bounded cohomology group
$H^2_b(\Gamma,\Z(P)_{sa})$ vanishes. Then, given $\eps>0$, there exists $\delta>0$
such that for a normal unital representation $\iota:M\rightarrow\mathcal B(\Hs)$
and a von Neumann subalgebra $N\subseteq \mathcal B(\Hs)$ with
$d(\iota(M),N)<\eps$, there exists a surjective $^*$-isomorphism
$\theta:M\rightarrow N$ with $\|\iota-\theta\|<\delta$.
\end{enumerate}
\end{theorem}

In order to distinguish the slightly different external rigidity properties
arising in Theorem \ref{TA} and the different parts of Theorem \ref{TB} above,
we call algebras satisfying the conclusion of Theorem \ref{TA} \emph{strongly
Kadison-Kastler stable}, algebras satisfying the conclusion of Theorem \ref{TB}
part (\ref{TB1}) \emph{weakly Kadison-Kastler stable} and algebras satisfying
the conclusion of Theorem \ref{TB} part (\ref{TB2}) \emph{Kadison-Kastler
stable}. With this terminology, the appropriate forms of the Kadison-Kastler
conjecture are that von Neumann algebras are strongly Kadison-Kastler stable and
separable $C^*$-algebras are Kadison-Kastler stable.

Part (\ref{TB1}) of Theorem \ref{TB} applies when $\Gamma$ is a free group
$\mathbb F_r$, $2\leq r\leq \infty$,  as these groups have cohomological
dimension one, so $H^2(\Gamma,L^2(\Z(P)_{sa}))=0$. In particular the approximate free group
factors, introduced in \cite{OP:Ann} as the first class of factors
containing a unique Cartan masa up to unitary conjugacy, have the form
$L^\infty(X)\rtimes_\alpha\mathbb F_r$ for some free ergodic measure preserving
profinite action $\alpha$. Consequently, these factors are weakly Kadison-Kastler
stable by Part (\ref{TB1}) of Theorem \ref{TB}.  As shown in
\cite[Section 5]{OP:Ann}, there are uncountably many pairwise nonisomorphic
factors in this class, including examples with property $\Gamma$. These latter
examples are Kadison Kastler stable by Part \ref{TB2} of Theorem \ref{TB}.  

The key strategy used to prove Theorem \ref{TA} is to replicate the crossed product structure of $L^\infty(X)\rtimes_\alpha SL_n(\mathbb Z)$ inside a nearby factor $N$. One can transfer the copy of $L^\infty(X)$ into $N$ using an embedding theorem of EC from \cite{C:Acta} (Theorem \ref{Injective} (\ref{Injective:Part1}) below), and transfer normalizers of $L^\infty(X)$ from $L^\infty(X)\rtimes_\alpha SL_n(\mathbb Z)$ to $N$ (see Section \ref{NormSection}).  To show that $N$ is generated by the copy of $L^\infty(X)$ and its normalizers we work at the level of Hilbert space by transferring the problem to the situation where both factors are in standard form. We do this in Section \ref{Cartan}, which provides a general reduction procedure for weak-Kadison Kastler stability and should be of more general use  (forthcoming work will show how this method can be used to transfer a number of structural properties between close II$_1$ factors).  The resulting factor $N$ will then be a twisted crossed product $L^\infty(X)\rtimes_{\alpha ,\omega} SL_n(\mathbb Z)$ arising from the original action and with a unitary valued $2$-cocycle which is uniformly close to the identity operator; the cohomology assumptions of Theorems \ref{TA} and \ref{TB} are used to ensure that this $2$-cocycle vanishes, so $N\cong L^\infty(X)\rtimes_\alpha SL_n(\mathbb Z)$.

The tensor factor $R$ in Theorem \ref{TA} ensures that $(L^\infty(X)\rtimes_\alpha
SL_n(\mathbb Z))\vnotimes R$ has Kadison's similarity property (a consequence of strong Kadison-Kastler stability for II$_1$ factors, see \cite{CCSSWW:InPrep}) and ensures that our resulting isomorphism is spatial.  To be able to work with the subfactor $L^\infty(X)\rtimes_\alpha SL_n(\mathbb Z)\subseteq (L^\infty(X)\rtimes_\alpha
SL_n(\mathbb Z))\vnotimes R$ we examine McDuff factors  (those absorbing a copy of $R$ tensorially) in Section \ref{properties}, where we show that a factor sufficiently close to a  McDuff factor is itself McDuff. Moreover, after making a small unitary perturbation, it is possible to identify a common tensor factor of $R$ in both algebras, whose tensorial complements are close.  Thus we can ``remove'' the copy of $R$ from $(L^\infty(X)\rtimes_\alpha
SL_n(\mathbb Z))\vnotimes R$ and apply our early work on normalizers to an identified algebra close to $L^\infty(X)\rtimes_\alpha SL_n(\mathbb Z)$. In Section \ref{KKStable} we assemble the proofs of Theorem \ref{TA} and \ref{TB}, and obtain the additional information needed for strong Kadison-Kastler stability when the bounded group cohomology vanishes. An extended outline of the methods used to prove Theorems \ref{TA} and \ref{TB} can be found in the expository article \cite{CCSSWW:PNAS}.  

\renewcommand*{\thetheorem}{\roman{theorem}}

\numberwithin{theorem}{section}

\section{Preliminaries}\label{Prelim}

We begin by recalling the definition of the Kadison-Kastler metric from
\cite{KK:AJM} and near inclusions from \cite{C:Acta}, and the ``complete'' versions of these concepts, which are implicit in \cite{C:Acta}, and explicitly appear in \cite{CSSW:GAFA,Roy:arxiv}.
\begin{definition}\label{KKMetric}
\begin{enumerate}[(i)]\item Let $M$ and $N$ be von Neumann subalgebras of $\mathcal B(\Hs)$.  The
\emph{Kadison-Kastler distance} $d(M,N)$ between $M$ and $N$ is the infimum of
those $\gamma>0$ with the property (T)hat, given an operator $x$ in one of the
unit
balls of $M$ or $N$, there exists $y$ in the other unit ball with
$\|x-y\|<\gamma$.  The complete version of the metric is defined by
$
d_{cb}(M,N)=\sup_{n\geq 1}\,d(M\otimes \mathbb M_n,N\otimes\mathbb M_n).
$
\item For
$\gamma>0$, write $M\subseteq_\gamma N$ if each $x\in M$ can be approximated by
some $y\in N$ with $\|x-y\|\leq \gamma\|x\|$.  Write $M\subset_\gamma N$ when
there exists $\gamma'<\gamma$ with $M\subseteq_{\gamma'}N$. 
Similarly, we write $M\subseteq_{cb,\gamma}N$ if $M\otimes\mathbb
M_n\subseteq_{\gamma}N\otimes\mathbb M_n$ for all $n\in\mathbb N$ and
$M\subset_{cb,\gamma}N$ if there exists $\gamma'<\gamma$ with
$M\subseteq_{cb,\gamma'}N$.
\end{enumerate}
\end{definition}

The following easy estimate will be used repeatedly in the sequel.
\begin{equation}\label{NearTriangleUnitary}
M\subseteq_\gamma N,\quad u\in\mathcal U({\mathcal{B}}(\Hs))\implies
M\subseteq_{\gamma+2\|u-I\|} uNu^*.
\end{equation}
This is obtained as follows. For $x$ in the unit ball of $M$, choose $y\in N$ with
$\|x-y\|\leq\gamma$. Then $\|x-uyu^*\|\leq\|x-uxu^*\|+\|u(x-y)u^*\|\leq
2\|u-I\|+\gamma$. 

Note that if two von Neumann algebras $M$ and $N$ act degenerately on a Hilbert space $\Hs$, and have $d(M,N)$ small, it is easy to modify the situation so that $M$ and $N$ share the same unit.  Henceforth we assume all close von Neumann algebras contain the identity operator on the underlying Hilbert space. We incorporate this assumption into the definitions below.

\begin{definition}\label{KKRigid}
Let $M$ be a von Neumann algebra.  
\begin{enumerate}[(i)]
\item Say that $M$ is \emph{strongly Kadison-Kastler stable} if for all
$\eps>0$, there exists $\delta>0$ such that given any faithful unital normal
representation $M\subseteq\mathcal B(\Hs)$ and a von Neumann algebra
$N\subseteq\mathcal B(\Hs)$ containing $I_\Hs$ with $d(M,N)<\delta$, then there
exists a unitary operator $u$ on $\Hs$ with $uMu^*=N$ and $\|u-I_\Hs\|<\eps$.  
\item Say that $M$ is \emph{Kadison-Kastler stable} if there exists $\delta>0$
such that given a faithful unital normal representation $M\subseteq\mathcal
B(\Hs)$ and a von Neumann algebra $N\subseteq\mathcal B(\Hs)$ with $I_\Hs\in N$
such that $d(M,N)<\delta$, then there exists a unitary operator $u$ on $\Hs$
with $uMu^*=N$.
\item Say that $M$ is \emph{weakly Kadison-Kastler stable} if there exists
$\delta>0$ such that given a faithful unital normal representation
$M\subseteq\mathcal B(\Hs)$ and a von Neumann algebra $N\subseteq\mathcal B(\Hs)$
with $I_\Hs\in N$ such that $d(M,N)<\delta$, then $M$ and $N$ are
$^*$-isomorphic.
\end{enumerate}
\end{definition}

In \cite{C:Acta} it is observed that near inclusions behave better than the metric $d$ with respect to matrix amplifications and commutants (see Proposition \ref{LocalDK} below). For this reason, we
state technical results  using hypotheses of the form $M\subset_\gamma
N$ and $N\subset_\gamma M$, but formulate the main results of the paper  using
the Kadison-Kastler metric $d$. Spatial derivations provide a key ingredient in working with commutants of near inclusions: an operator $T\in\mathcal B(\Hs)$
induces a derivation $S\mapsto TS-ST$ on $\mathcal
B(\Hs)$ denoted $\ad(T)$. Arveson's distance formula from \cite{A:JFA} (see \cite[Proposition 2.1]{C:IJM} for the formulation we use) shows that for a von Neumann algebra $M\subseteq\mathcal B(\Hs)$ and $T\in\mathcal  B(\Hs)$, 
\begin{equation}\label{DForm}
d(T,M')=\frac{1}{2}\|\ad(T)|_M\|_{\cb}.
\end{equation}
This enables one to take commutants of complete near inclusions.

\begin{proposition}\label{cb-com}
Let $M,N\subseteq \mathcal B(\Hs)$ be von Neumann algebras acting nondegenerately on a Hilbert space
$\Hs$
with $M\subseteq_{\cb,\gamma}N$. Then $N'\subseteq_{\cb,\gamma}M'$ and for any Hilbert space $\Ks$, $M\,\overline{\otimes}\,\mathcal B(\Ks)\subseteq_{\cb,\gamma}N\,\overline{\otimes}\,\mathcal B(\Ks)$. In particular,
if $d_{\cb}(M,N)\leq \gamma$, then $d_{\cb}(M',N')\leq 2\gamma$.
\end{proposition}
\begin{proof}
Given $n\in\mathbb N$ and $T\in\mathcal B(\Hs\otimes \mathbb C^n)$, 
(\ref{DForm}) shows that $d(T,M'\otimes \mathbb M_n)=\|\ad(T)|_{M\otimes \mathbb CI_n}\|_{\cb}/2$.
Let
$T\in N'\otimes\mathbb M_n$.  Given $x\in M\otimes \mathbb M_n\otimes \mathbb
M_s$, choose $y\in N\otimes \mathbb M_{n}\otimes\mathbb M_s$ with
$\|x-y\|\leq\gamma\|x\|$. Then
\begin{equation}
\|(\ad(T)|_{M\otimes\mathbb M_n}\otimes\id_{\mathbb M_s})(x)\|=\|\ad(T\otimes
I_s)(x)\|
\leq 2\|T\otimes I_s\|\|x-y\|\leq 2\|T\|\gamma\|x\|.
\end{equation}
This holds for $x\in M\otimes \mathbb CI_n\otimes\mathbb M_s$,
 so $d(T,M'\otimes\mathbb M_n)\leq \|T\|\gamma$.  Hence $N'\subseteq_{\cb,\gamma}M'$.  For the second claim, take commutants of the near inclusion $N'\otimes\mathbb CI_{\Ks}\subseteq_{\cb,\gamma}M'\otimes\mathbb CI_{\Ks}$.
The final statement is immediate.
\end{proof}

The similarity property for $M$ is characterized by the ability  to take commutants of near inclusions $M\subseteq_\gamma N$ with constants independent of the underlying Hilbert space,  (see \cite{CCSSWW:InPrep}). The most convenient form of the similarity problem for this purpose is property $D_k$: a $C^*$-algebra $A$ has property $D_k$ for some $k>0$ if, for every faithful  unital representation $\pi:A\rightarrow
\mathcal B(\Hs)$, we have $d(T,\pi(A)')\leq k\|\ad(T)|_{\pi(A)}\|$ for $T\in \mathcal B(\Hs)$. The existence of some $k>0$ such that $A$ has property $D_k$ is equivalent to the similarity property by \cite{Ki:JOT}. The proposition below records a strengthening of the commutation result from \cite[Theorem 3.1]{C:Acta}.
\begin{proposition}\label{DK}
Let $M$ and $N$ be von Neumann algebras acting nondegenerately on $\Hs$ and suppose that
$M\subseteq_\gamma N$ for some $\gamma >0$. 
\begin{enumerate}[(i)]
\item\label{DK:Part1} If $k >0$ and $M$ has property $D_k$, then
$N'\subseteq_{\cb,2k\gamma}M'$ and
$M\subseteq_{\cb,2k\gamma}N$.  
\item\label{DK:Part2} If $M$ is a II$_1$ factor with property $\Gamma$, then 
$N'\subseteq_{\cb,5\gamma}M'$ and
$M\subseteq_{\cb,5\gamma}N$.
\end{enumerate}
\end{proposition}
\begin{proof}
(\ref{DK:Part1}). Fix $n\in\mathbb N$ and $T\in N'\otimes\mathbb M_n\subseteq\mathcal
B(\Hs\otimes\mathbb C^n)$. Applying property $D_k$ to the amplification
$M\otimes I_n$ acting on $\Hs\otimes\mathbb C^n$, we see that
$d(T,M'\otimes\mathbb M_n)\leq k\|\ad(T)|_{M\otimes I_n}\|\leq 2k\gamma\|T\|$
(where the last estimate arises from $M\subseteq_\gamma N$ just as in
\cite[Theorem 3.1]{C:Acta}).  Thus $N'\subseteq_{\cb,2k\gamma}M'$ and so
$M\subseteq_{\cb,2k\gamma}N$ by Proposition \ref{cb-com}.  

(\ref{DK:Part2}). This now follows
immediately as II$_1$ factors with property $\Gamma$ have property $D_{5/2}$ by combining \cite[Theorem 13]{Pi:IJM} and \cite[Remark 4.7]{Pi:StP}.\footnote{One can obtain a smaller value of $k$ such that property $\Gamma$ factors have property $D_{k}$ by modifying the methods in \cite{C:JFA}, but since this is not available in the literature we use property $D_{5/2}$ here.} 
\end{proof}

Property $D_k$ can also be used to show that isomorphisms close to the identity are necessarily spatially implemented.  The first part of the
lemma below is obtained by making minor changes to the proof of \cite[Proposition
3.2]{C:IJM} (which handles the case of properly infinite von Neumann algebras
using property $D_{3/2}$, and McDuff factors using property $D_{5/2}$).  As
property $\Gamma$ factors have property $D_{5/2}$, (ii) is an
immediate consequence of (i).

\begin{lemma}[{\cite{C:IJM}}]\label{IsomorphismDK}
Let $M$ be von Neumann algebra acting nondegenerately on $\Hs$ and suppose that $\theta:M\rightarrow\mathcal B(\Hs)$ is a $^*$-homomorphism with $\|\theta(x)-x\|\leq \gamma\|x\|$ for $x\in M$.
 
\begin{enumerate}[(i)]
\item\label{IsomorphismDK:Part1} Suppose that $M$ has property $D_k$ for some
$k\geq 1$ and that $\gamma<1/k$. Then there exists a unitary $u$ on $\Hs$ such
that $\theta=\Ad(u)$ and 
\begin{equation}
\|I_\Hs-u\|\leq 2^{1/2}k\gamma \left(1+(1-(k\gamma)^2)^{1/2}\right)^{-1/2}\leq 2^{1/2}k\gamma.
\end{equation}
\item\label{IsomorphismDK2} Suppose that $M$ has property $\Gamma$ and
that $\gamma<2/5$. Then there exists a unitary $u$ on $\Hs$ such that
$\theta=\Ad(u)$ and $\|I_\Hs-u\|\leq 2^{-1/2}5\gamma$.
\end{enumerate}
\end{lemma}

Without the similarity property, one can also take commutants when $M$ acts on a given Hilbert space $\Hs$ with a finite set of $m$ cyclic vectors for $\Hs$; the resulting estimates depend on $m$ (\cite{C:Scand}). We give an alternative proof of this fact below, which significantly improves the constants involved.  We start by isolating a technical lemma.

\begin{lemma}\label{cb-derivlem}
Let $M\subseteq\mathcal B(\Hs)$ be a von Neumann algebra such that $\Hs$ has a
cyclic vector for $M$.  Then, for every derivation $\delta:M\rightarrow\mathcal
B(\Hs)$, we have
$
\|\delta\|_{\cb}\leq 2\|\delta\|_\row,
$
where $\|\delta\|_\row$ denotes the \emph{row norm} of $\delta$, given by
$\|\delta\|_\row=\sup_{n,r}\,\|\delta_{1\times n}(r)\|$, where the supremum is taken
over contractions $r\in \mathbb{M}_{1\times n}(M)$.
\end{lemma}
\begin{proof}
Let $n\in\mathbb N$ and $x\in\mathbb M_n(M)$. Since $M$ has a cyclic vector for $\mathcal B(\Hs)$, we have
\begin{equation}\label{vlem1}
\|\delta_n(x)\|=\sup\,\{\|r\delta_n(x)\|:r\in \mathbb{M}_{1\times n}(M),\
\|r\|\leq 1\},
\end{equation}
by \cite[Lemma 2.4 (i)$\implies$(iv) and Theorem 2.7]{PSS:JFA}. For an 
 $r\in \mathbb{M}_{1\times n}(M)$  with 
$\|r\|= 1$, the relation
$r\delta_n(x)=\delta_{1\times n}(rx)-\delta_{1\times n}(r)x$ gives 
$\|r\delta_n(x)\|\leq 2\|\delta\|_{\row}\|x\|$ and so 
$\|\delta\|_{\cb}\leq 2\|\delta\|_\row.$
\end{proof}

In the proof below we use \cite[Proposition 4.2]{SS:AJM}. We take this
opportunity to correct an oversight in the statement of this result, which
omitted the hypothesis that  $M$ is finite (which is required in order to appeal to \cite[Theorem 2.3]{CPSS:MA}). 
\begin{lemma}\label{cb-deriv}
Let $M\subseteq \mathcal{B}(\Hs)$ be a von Neumann algebra with a cyclic set of
$m$ vectors (i.e. there exist $\xi_1,\ldots,\xi_m\in\Hs$ with
$\overline{\mathrm{span}}\{x_i\xi_i:x_i\in M\,\ 1\leq i\leq m\}=\Hs$). If
$\delta:M\to \mathcal{B}(\Hs)$ is a bounded derivation, then $\delta$ is
completely bounded and
$\|\delta\|_{cb}\leq 2(1+\sqrt{2})m\|\delta\|$.
\end{lemma}

\begin{proof}
The key case is when $m=1$ and $M$ has a separable predual.  By
\cite[Corollary 3.2]{C:Scand} and (\ref{DForm}) it
suffices to prove the result for  $\delta$  of the form $\ad(T)|_M$ for
a fixed $T\in\mathcal B(\Hs)$.  

Let $p$ be the central projection in $M$ such that $M_0=Mp$ is finite and
$M_1=M(1-p)$ is properly infinite.  From \cite{P:Invent},  choose an amenable von Neumann algebra
$P_0\subseteq M_0$ with $P_0'\cap M_0=\Z(M_0)$ (when $M_0$ is a II$_1$ factor
this is \cite[Corollary 4.1]{P:Invent}, the extension to the case when $M_0$ is
II$_1$ can be found in \cite[Theorem 8]{SS:Scand}, and the general case is
obtained by splitting $M_0$ as a direct sum of its II$_1$ part and its finite
type I part, which is already amenable).  Let $P_1$ be a properly infinite
amenable von Neumann subalgebra of $M_1$ and let $P=(P_0\cup P_1)''$, which is
an amenable subalgebra of $\mathcal B(\Hs)$ containing $p$.  

Since $P$ is amenable, we can find 
$S\in \overline{\mathrm{co}}^{w^*}\{uTu^*:u\in\U(P)\}\cap P'$.  By construction (see \cite[Section 2]{C:IJM}) $\|\ad(S)|_M\|\leq\|\ad(T)|_M\|$ and $\|S-T\|\leq\|\ad(T)|_M\|$.  Further, just as in \cite[Theorem 2.4]{C:IJM}, $\{S(1-p),(1-p)S^*\}'\cap M_1$ is properly infinite so
$\|\ad(S(1-p))|_{M_1}\|=2d(S(1-p),M_1')=\|\ad(S(1-p))|_{M_1}\|_\cb$
by \cite[Corollary 2.2]{C:IJM}. The map $\ad(Sp)|_{M_0}$ is a $P_0$-module map $M_0\rightarrow\mathcal B(p(\Hs))$ and so by \cite[Theorem 4.2]{SS:AJM}, 
$\|\ad(Sp)|_{M_0}\|_\row\leq \sqrt{2}\|\ad(Sp)|_{M_0}\|$.
It follows that $ \|\ad(S)|_M\|_\row\leq \sqrt{2}\|\ad(S)|_M\|$.
Then $\|\ad(S)|_M\|_\cb\leq2\sqrt{2}\|\ad(S)|_M\|\leq 2\sqrt{2}\|\ad(T)|_M\|$, by Lemma \ref{cb-derivlem}.
Since $\|T-S\|\leq\|\ad(T)|_M\|$ and $\ad (T)|_M=\ad(S)|_M+\ad(T-S)|_M$, we have
\begin{equation}
\|\ad(T)|_M\|_\cb
\leq 2\sqrt{2}\|\ad(T)|_M\|+2\|T-S\|\leq 2(1+\sqrt{2})\|\ad(T)|_M\|,
\end{equation}
proving the result in the case when $m=1$ and $M$ has separable predual. 
A standard argument can be used to extend this estimate to the situation when $M_*$ is nonseparable (this can be found in the preprint version of this paper on the arXiv). A cyclic set of $m$ vectors can be reduced to the
case of a single cyclic vector by tensoring by $\mathbb{M}_m$, see
\cite[Proposition 2.6]{CSSW:GAFA}.
\end{proof}

In the context of near inclusions, the previous lemma has the following form.
\begin{proposition}\label{LocalDK}
Let $M$ be a von Neumann algebra acting on a Hilbert space $\Hs$ and suppose
that $M$ has a finite cyclic set of $m$ vectors for $\Hs$.  Then:
\begin{enumerate}[(i)]
\item\label{LocalDK:Part1}  For  $T\in\mathcal B(\Hs)$, we have
$
\inf\{\|T-S\|:S\in M'\}\leq (1+\sqrt{2})m\|\ad(T)|_M\|.
$
\item\label{LocalDK:Part2}  Given another von Neumann algebra $N$ on $\Hs$ with
$M\subseteq_\gamma N$, we have $N'\subseteq_{2(1+\sqrt{2})m\gamma}M'$.
\end{enumerate}
\end{proposition}

\begin{proof} Part (\ref{LocalDK:Part1}) follows from Lemma \ref{cb-deriv} and Arveson's distance formula in (\ref{DForm}) above.  For (\ref{LocalDK:Part2}), note that if $N$ is another von Neumann algebra on $\Hs$ with $M\subseteq_\gamma N$, then for $T\in N'$ the near inclusion $M\subset_\gamma N$ gives the estimate $\|\ad(T)|_M\|\leq 2\|T\|\gamma$ (see  \cite[Theorem 3.1]{C:Acta}). 
\end{proof}

A II$_1$ factor $M$ is said to be in \emph{standard position} on  a Hilbert space $\Hs$, when $M'$ is finite and there  is a vector $\xi\in\Hs$
which is tracial for $M$ and $M'$, i.e.
the vector state $\langle \cdot\xi,\xi\rangle$ gives the traces $\tau_M$ and $\tau_{M'}$ on $M$ and
$M'$ respectively. Such a vector $\xi$ is cyclic for both $M$ and $M'$, and defines a conjugate linear isometry $J$ (or $J_M$ when necessary) by extending $J(m\xi) = m^*\xi$ to $\Hs$. This is called the conjugation operator, and has
the property (T)hat $JMJ=M'$. When $M$ is in standard position on $\Hs$, any trace vector $\xi$ for $M$ is automatically tracial for $M'$.

If $M$ is in standard position on $\Hs$ with a specified tracial vector
$\xi$ and $P\subseteq M$ is any von Neumann
subalgebra, then $e_P$
denotes the projection onto $\overline{P\xi}$.
The basic construction is then defined to be $(M\cup\{e_P\})''$ and is denoted by
$\langle M,e_P\rangle$. This
algebra dates back to \cite{Sk:JFA} and \cite{C:MA}, and was first used
systematically in the work of Jones \cite{J:Invent}. We list a few standard properties of the basic construction below (see \cite{J:Invent} or \cite{JSun:Book}) and record a technical lemma.

\begin{properties}\label{bcon-prop}
With the notation above:
\begin{enumerate}[(i)]
\item\label{bcon-prop:Item1}
 $e_P=Je_PJ$;
\item\label{bcon-prop:Item2}
 $e_P$ commutes with $P$ and $P=M\cap \{e_P\}'$;
\item\label{bcon-prop:Item3} $\langle M,e_P\rangle '=J_MPJ_M$;
\item\label{bcon-prop:Item4} The map $p\mapsto pe_P$, $p\in P$, is an algebraic
isomorphism,
and
consequently isometric;
\item\label{bcon-prop:Item5} For each $x\in M$, $e_Pxe_P=E^M_P(x)e_P$, where
$E^M_P$ denotes the unique trace preserving conditional expectation
from $M$ onto $P$, and $e_P\langle M,e_P\rangle e_P=Pe_P$.
\end{enumerate}
\end{properties}
\begin{lemma}\label{basiccon-lem}
Let $M$ be a   {\rm{II}}$_1$ factor with separable predual, in standard
position on  $\Hs$ with tracial vector $\xi$. Let $P$ be a von Neumann subalgebra of $M$ satisfying $P'\cap M\subseteq P$. Given a unitary $v\in P'\cap \langle M,e_P\rangle$, there
exists a unitary $u\in
P'\cap M=\Z(P)$ such that $v\xi =u\xi$.
\end{lemma}

\begin{proof}
 By \cite[Lemma 3.2]{FaSWWig:JFA}, $e_P\in \mathcal Z(P'\cap \langle M,e_P\rangle)$, so $ve_P=e_Pve_P\in Pe_P$. Write $ve_P=ue_P$ for some $u\in P$. Thus $v\xi=u\xi$. Then $\|u\|\leq 1$ and $\langle (I_\Hs -u^*u)\xi,\xi\rangle =0$ so $u\in \mathcal U(P)$. Since $v,e_P\in P'$, 
\begin{equation}\label{generate.4}
(uy-yu)e_P=(ue_Py-yve_P)=ve_Py-yve_P=0,\quad y\in P.
\end{equation}
By Properties \ref{bcon-prop} (\ref{bcon-prop:Item4}), $u\in P'\cap P=\Z(P)$.
\end{proof}

Given a nonzero vector $\eta\in\Hs$, write
$e^M_\eta$ for the projection in $M'$ onto the subspace $\overline{M\eta}$ and
$e^{M'}_\eta$ for the projection in $M$ onto $\overline{M'\eta}$.  If $M'$ is
also a finite factor, with faithful normalized trace $\tau_{M'}$, then the
quantity $\dim_M(\Hs)=\tau_M(e^{M'}_\eta)/\tau_{M'}(e^M_\eta)$ is independent of
the choice of nonzero $\eta\in\Hs$ (see \cite[Part III Chapter 6]{Dix:Book}
 for example). When $M'$ is a   II$_\infty$ factor, set
$\dim_M(\Hs)=\infty$.  We recall some properties of $\dim_M(\Hs)$ from
\cite[Part III Chapter 6]{Dix:Book}.

\begin{properties}\label{CC}
Let $M$ be a   ${\mathrm{II}}_1$ factor acting nondegenerately on $\Hs$.
\begin{enumerate}[(i)]
\item\label{CC:Item1} If $p'$ is a projection in $M'$, then
$\dim_{Mp'}(p'\Hs)=\tau_{M'}(p')\dim_{M}(\Hs)$.
\item\label{CC:Item2}If $p$ is a projection in $M$, then
$\dim_{pMp}(p\Hs)=\dim_M(\Hs)/\tau_M(p)$.
\item\label{CC:Item3} $M$ acts in standard position on $\Hs$ if and only if
$\dim_M(\Hs)=1$.
\item\label{CC:Item4} If $\dim_M(\Hs)\geq 1$, then there is a tracial vector
$\xi\in\Hs$ for $M$.
\item\label{CC:Item5} If $\dim_M(\Hs)\leq m\in\mathbb N$, then there is a set of
$m$-cyclic vectors for $M$ on $\Hs$.\end{enumerate}
\end{properties}

We need some standard results for approximating unitaries and projections in the
sequel. Lemma \ref{ProjLem} is a consequence of \cite[Lemma 1.10]{Kh:JOT} (it
follows by noting that the function $\alpha(t)$ used there satisfies
$\alpha(t)\leq \sqrt{2}t$ for $0\leq t<1$), and Lemma \ref{ProjLem2} is the usual
estimate in the Murray-von Neumann equivalence of close projections (see
\cite[Lemma 6.2.1]{Mur:Book} for example).
\begin{lemma}\label{ProjLem}
Suppose that $\gamma<1$ and let $M\subset_\gamma N$ be a near inclusion of von Neumann algebras sharing the same unit.
\begin{enumerate}[(i)]
\item\label{ProjLem:Part1} Given a unitary $u\in M$, there exists a unitary $v\in N$ with
$\|u-v\|<\sqrt{2}\gamma$.
\item\label{ProjLem:Part2} Given a projection $p\in M$, there exists a projection $q\in N$
with $\|p-q\|<2^{-1/2}\gamma$.
\end{enumerate}
\end{lemma}

\begin{lemma}\label{ProjLem2}
Let $p$ and $q$ be projections in a von Neumann algebra $M$ with $\|p-q\|<1$.
Then there exists a unitary $u\in M$ with $upu^*=q$ and $\|u-I_M\|\leq
\sqrt{2}\|p-q\|$.
\end{lemma}

Lemma 3.6 of \cite{CSSW:GAFA} examined the center-valued traces on close finite
von Neumann algebras. The next lemma gives improved estimates in the special
case of close II$_1$ factors.

\begin{lemma}\label{MasaTr}
Let $M$ and $N$ be   ${\mathrm{II}}_1$ factors acting nondegenerately on a
Hilbert space $\Hs$ and
satisfying
$M\subset_{\gamma}N\subset_{\gamma}M$ for a constant $\gamma<2^{-3/2}$. Let
$P\subseteq M\cap N$ be a diffuse von Neumann algebra containing $I_{\Hs}$. Then
$\tau_M|_P=\tau_N|_P$.
\end{lemma}
\begin{proof}
The result will follow if we can show that $\tau_M(p_n)=\tau_N(p_n)$ for all projections $p_n\in P$ with $\tau_M(p_n)=2^{-n}$,
$n\geq 0$. The case $n=0$ is clear, so inductively take a projection $p_{n+1}\in P$ with $\tau_M(p_{n+1})=2^{-(n+1)}$ and choose a projection $p_n\in P$ so that $p_{n+1}\leq p_n$ and $\tau_M(p_n)=2^{-n}$. Choose $v\in \mathcal U(M)$ so that $vp_{n+1}v^*=p_n-p_{n+1}$. By Lemma \ref{ProjLem} (\ref{ProjLem:Part1}), take $u\in \mathcal U(N)$ so that $\|v-u\|\leq \sqrt{2}\gamma<2^{-1}$. Then $\|vp_{n+1}v^*-up_{n+1}u^*\|\leq 2\|v-u\|< 1$. Thus $up_{n+1}u^*$ and $p_n-p_{n+1}$ are equivalent in $N$. Since the inductive hypothesis gives $\tau_N(p_n)=2^{-n}$, it follows that $\tau_N(p_{n+1})=2^{-(n+1)}=\tau_M(p_{n+1})$.
\end{proof}

\begin{lemma}\label{Trace}
Let $M$ and $N$ be   {\rm{II}}$_1$ factors acting nondegenerately on a
Hilbert space $\Hs$ with
$M\subset_\gamma N$ for some $\gamma<1$.  Let $\Phi$ be a state on $\mathcal B(\Hs)$ extending
$\tau_N$.  Then
\begin{equation}\label{L2.1}
|\tau_M(x)-\Phi(x)|\leq(2+2\sqrt{2})\gamma\|x\|\leq 5\gamma\|x\|,\quad x\in M.
\end{equation}
\end{lemma}
\begin{proof}
Fix $x$ in the unit ball of $M$ and a unitary $u\in M$. Choose $y\in N$ with
$\|x-y\|\leq\gamma$ and, by  Lemma \ref{ProjLem} (\ref{ProjLem:Part1}), a unitary
$v\in N$ with $\|u-v\|<\sqrt{2}\gamma$.  Then
\begin{align}
\|uxu^*-vyv^*\|&=\|uxu^*-vxv^*+v(x-y)v^*\|\leq\|v^*ux-xv^*u\|+\|x-y\|\notag\\
&\leq 2\|v^*u-I_\Hs\|+\|x-y\|
=2\|u-v\|+\|x-y\|\leq(2\sqrt{2}+1)\gamma,
\end{align}
and so, since $\Phi(vyv^*)=\tau_N(vyv^*)=\tau_N(y)=\Phi(y)$,
\begin{equation}
|\Phi(uxu^*-x)|\leq|\Phi(uxu^*-vyv^*)|+|\Phi(vyv^*-y)|+|\Phi(y-x)|
<(1+2\sqrt{2})\gamma+\gamma<5\gamma.
\end{equation} 
 By \cite[Theorem 8.3.5]{KRin:Book2},  $\tau_M(x)I_\Hs\in \overline{\mathrm{co}}^{\|\cdot\|}\{uxu^*:u\in \mathcal U(M)\}$, so   (\ref{L2.1})
follows.
\end{proof}

Close subalgebras of II$_1$
factors have close relative commutants (see \cite[Proposition 2.7]{C:IJM} which
states Lemma \ref{RelCom} when $M=N$). Recall that if $P$ is a unital
von Neumann subalgebra
of a
  II$_1$ factor $M$, then the unique $\tau_M$-preserving conditional
expectation $E_{P'\cap M}^M$ from $M$ onto $P'\cap M$ satisfies
the condition that $E_{P'\cap M}(x)$ is the unique element of minimal
$\|\cdot\|_2$-norm in 
${\overline{\mathrm{co}}}^{\|\cdot\|_2}\{uxu^*:u\in\mathcal U(P)\}\subseteq L^2(M)$ for each $x\in M$ (see
\cite[Lemma 3.6.5 (i)]{SS:Book2}, for example).

\begin{lemma}\label{RelCom}
Let $M$ and $N$ be   ${\mathrm{II}}_1$ factors acting
nondegenerately on a Hilbert space $\Hs$ and suppose
that $P\subseteq
M$
and $Q\subseteq N$ are unital von Neumann subalgebras.
\begin{enumerate}[(i)]
\item\label{RelCom:Part1} Suppose that $M\subseteq_\gamma N$ and
$Q\subseteq_\delta
P$. Then $P'\cap M\subseteq_{2\sqrt{2}\delta+\gamma}Q'\cap N$.
\item\label{RelCom:Part2}  Suppose that $M\subseteq_{\cb,\gamma} N$ and
$Q\subseteq_{\delta}
P$. Then $P'\cap M\subseteq_{\cb,2\sqrt{2}\delta+\gamma}Q'\cap N$.
\end{enumerate}
\end{lemma}
\begin{proof}
(\ref{RelCom:Part1}).  Given $x\in P'\cap M$ with $\|x\|\leq 1$ choose $y\in N$ with
$\|x-y\|\leq\gamma$.  For a unitary $v\in Q$ use Lemma \ref{ProjLem} (\ref{ProjLem:Part1}) to find
a unitary $u\in P$ with $\|u-v\|\leq\sqrt{2}\delta$.  Noting that $uxu^*=x$ and that $\|vxv^*-uxu^*\|\leq 2\|u^*v-I_\Hs\|\leq 2\sqrt{2}\delta$, we
obtain the estimate
\begin{equation}
\|vyv^*-x\|\leq \|vyv^*-vxv^*\|+\|vxv^*-uxu^*\|
\leq \gamma+2\sqrt{2}\delta.
\end{equation}
Since $E_{Q'\cap N}(y)$ is a strong operator limit of convex
combinations of elements $vyv^*$ for unitaries $v\in Q$, it follows that
$\|E_{Q'\cap N}(y)-x\|\leq 2\sqrt{2}\delta+\gamma$, as
required.

(\ref{RelCom:Part2}). Fix $n\in\mathbb N$. We have $P\otimes \mathbb
CI_n\subseteq M\otimes\mathbb M_n$ and $Q\otimes\mathbb CI_n\subseteq
N\otimes\mathbb M_n$ and the near inclusions $M\otimes\mathbb
M_n\subseteq_{\gamma}N\otimes\mathbb M_n$ , $Q\otimes\mathbb
CI_n\subseteq_\delta P\otimes\mathbb CI_n$.  By part (\ref{RelCom:Part1}) we
have 
\begin{equation}
(P'\cap M)\otimes\mathbb
M_n=(P\otimes\mathbb CI_n)'\cap (M\otimes\mathbb
M_n)\subseteq_{2\sqrt{2}\delta +\gamma}(Q\otimes\mathbb
CI_n)'\cap(N\otimes\mathbb M_n)=(Q'\cap N)\otimes\mathbb M_n.
\end{equation} 
Since $n$
was arbitrary, the  near inclusion $P'\cap
M\subseteq_{\cb,2\sqrt{2}\delta+\gamma}Q'\cap N$ holds.
\end{proof}

The following immediate consequence of Lemma \ref{RelCom} will be used repeatedly in the sequel.
\begin{lemma}\label{relcom2}
Let $M$ and $N$ be II$_1$ factors acting nondegenerately on a Hilbert
space $\Hs$ and suppose that 
$M\subset_{\gamma}N\subset_{\gamma}M$ where $\gamma>0$. Let $P\subseteq M\cap N$ be
a von Neumann subalgebra. Provided $\gamma<1$, we have $P'\cap M\subseteq P$ if and only if $P'\cap N\subseteq P$. In particular $P$ is a masa in $M$ if and only if it is a masa in $N$.
\end{lemma}

\begin{proof}
We have $P'\cap N\subseteq_\gamma P'\cap M$ by Lemma \ref{RelCom} (\ref{RelCom:Part1}), so when $P'\cap M=\Z(P)\subseteq P$, we have $P'\cap N\subseteq_\gamma \Z(P)\subseteq P'\cap N$. As $\gamma<1$, $P'\cap N=\Z(P)$ (see \cite[Proposition 2.4]{CSSWW:Acta}).
\end{proof}

Next, we set out our notation for twisted crossed products. Every action
$\alpha:\Gamma\curvearrowright P$ of a countable discrete group $\Gamma$ on a
von Neumann algebra $P$ with a separable predual restricts to an action of $\Gamma$ on the abelian
group $\mathcal U(\Z(P))$ of unitaries in the center of $P$. Concretely, a \emph{$2$-cocycle} is a 
function $\omega:\Gamma\times \Gamma\rightarrow\U(\Z(P))$ satisfying 
\begin{equation}
\alpha_g(\omega(h,k))\omega(gh,k)^*\omega(g,hk)\omega(g,h)^*=I_P,\quad g,h,k\in
\Gamma.
\end{equation}
Two $2$-cocycles $\omega$ and $\omega'$ are cohomologous if they differ by a
coboundary, in the sense that there exists $\nu:\Gamma\rightarrow\U(\Z(P))$
with 
\begin{equation}\label{cohomeqn}
\omega'(g,h)=\alpha_g(\nu(h))\nu(gh)^*\nu(g)\omega(g,h),\quad g,h\in \Gamma.
\end{equation}
The second cohomology group $H^2(\Gamma,\U(\Z(P)))$ consists of the equivalence
classes of the group $Z^2(\Gamma,\U(\Z(P)))$ of $2$-cocycles under the relation of
being cohomologous. A $2$-cocycle $\omega$ is \emph{normalized} if
$\omega(g,h)$ is trivial when either $g=e$ or $h=e$. Every $2$-cocycle is
cohomologous to a normalized $2$-cocycle, so there is no loss of generality in
restricting to normalized cocycles.

Given a $2$-cocycle $\omega\in Z^2(\Gamma,\U(\Z(P)))$, the twisted crossed product
von Neumann algebra $P\rtimes_{\alpha,\omega}\Gamma$ is a von Neumann algebra
generated by a unital copy of $P$ and unitaries $(u_g)_{g\in \Gamma}$ for which there
is a faithful normal conditional expectation $E:P\rtimes_{\alpha,\omega}\Gamma\rightarrow P$ and
for which the following conditions hold:
\begin{equation}\label{TCP}
u_gxu_g^*=\alpha_g(x),\quad E(u_g)=0\ (g\neq e),\quad u_gu_h=\omega(g,h)u_{gh},\quad x\in P,\ g,h\in\Gamma.
\end{equation}
 If, in addition, $\omega$ is normalized, then $u_e$ is the identity operator.  The crossed product is usually constructed concretely starting from a
faithful representation of $P$, but for our purposes all that matters is that these algebras are
 characterized by the first two conditions in (\ref{TCP}) and the cocycle is given by the third condition (see \cite{Ch:MJ,Su:PRIMS2} for example). The isomorphism class of the crossed product only depends on the
cohomology class of the cocycle $\omega$. We record these facts in the next proposition.
\begin{proposition}\label{TCPProp}Let $\alpha:\Gamma\curvearrowright P$ be a trace-preserving action of  a countable discrete group on
a finite von Neumann algebra $P$ with a separable predual and let $\omega\in
Z^2(\Gamma,\U(\Z(P)))$.
\begin{enumerate}[\rm (i)]
\item\label{TCPProp2} Suppose that a  finite von Neumann algebra $M$  is
generated by a unital copy of $P$  and
unitaries $(u_g)_{g\in \Gamma}$ with $u_e=I_P$ and
there is a faithful normal trace-preserving expectation
$E^M_P:M\rightarrow
P$ so that the first two conditions in (\ref{TCP}) hold and $u_gu_h\in\Z(P)u_{gh}$ for $g,h\in\Gamma$ (this last condition is automatic if $P'\cap M\subseteq P$). Then $M$ is $^*$-isomorphic
to $P\rtimes_{\alpha,\omega}\Gamma$ where $\omega$ is a normalized $2$-cocycle given by the third condition in 
(\ref{TCP}). Further, an isomorphism can be found which identifies the two
copies of $P$ and maps the unitaries $u_g$ in $M$ to the canonical unitaries in
the crossed product $P\rtimes_{\alpha,\omega}\Gamma$.
\item\label{TCPProp:Part2} If $\omega,\omega'\in Z^2(\Gamma,\U(\Z(P)))$ are cohomologous, then
$P\rtimes_{\alpha,\omega}\Gamma\cong P\rtimes_{\alpha,\omega'}\Gamma$.
\end{enumerate}
\end{proposition}

Our focus is on crossed products $M=P\rtimes \Gamma$ which are II$_1$
factors with $P'\cap M\subseteq P$. There are two
well known sets of conditions which lead to this situation.
\begin{enumerate}
\item The twisted version of Murray and von Neumann's classical
group-measure
space construction from \cite{MvN:1}. Take $P=L^\infty(X,\mu)$ to be an abelian von
Neumann algebra and a free
ergodic, probability measure preserving action $\Gamma\curvearrowright(X,\mu)$ inducing the action $\alpha:\Gamma\curvearrowright P$.  The conditions on the action ensure that $P\rtimes_{\alpha,\omega} \Gamma$ is a 
  II$_1$ factor and $P$ is a maximal abelian subalgebra of the crossed product for any $2$-cocycle $\omega$.
 \item When $P$ is a   II$_1$ factor and $\alpha:\Gamma\curvearrowright P$ is an
outer
action (in the sense that $\alpha_g$ is not inner for $g\neq e$), then the
crossed product $P\rtimes_\alpha \Gamma$ is again a   II$_1$ factor and $P$ is
an
irreducible subfactor of the crossed product, \cite{NT:PJA}, for any $2$-cocycle $\omega$.
\end{enumerate}
To unify these situations, say that an action $\alpha:\Gamma\curvearrowright P$ is
\emph{properly outer} if, for $g\in \Gamma$ with $g\neq e$ and a nonzero
projection
$z\in\Z(P)$ with $\alpha_g(z)=z$, the automorphism of $Pz$ induced by $\alpha_g$
is not inner.  If the
fixed point algebra of the restriction of $\alpha$ to $\Z(P)$ is trivial, then
we say that the action is {\emph{centrally ergodic}}. When $P$ is abelian, these
conditions reduce to freeness and ergodicity, while if $\Z(P)=\mathbb C1$ they
reduce to outerness.  Further, it is folklore that the resulting crossed product factors $M$ have $P'\cap M\subseteq P$.
\begin{proposition}\label{TCPFactorProp}
Let $P$ be a finite von Neumann algebra with a fixed faithful normal normalized
trace $\tau_P$ and suppose that $\alpha:\Gamma\curvearrowright P$ is 
a trace preserving, centrally ergodic, and properly outer action of a 
discrete group $\Gamma$. Then, for any $\omega\in Z^2(\Gamma,\U(\Z(P)))$, the twisted
crossed product $M=P\rtimes_{\alpha,\omega}\Gamma$ is a   ${\mathrm{II}}_1$
factor
with $P'\cap M\subseteq P$.
\end{proposition}

We end the section by examining some aspects of bounded group cohomology. Let $\alpha:\Gamma\curvearrowright P$ be a properly outer, centrally ergodic, trace preserving action of a countable discrete group on a finite von Neumann algebra with separable predual.  Then $\alpha$ naturally induces actions of $\Gamma$ on $\Z(P)_{sa}$ and its real Hilbert space completion $L^2(\Z(P))_{sa}$. These are coefficient $\Gamma$-modules in the language of \cite{M:Book}, giving rise to the bounded cohomology groups $H^2_b(\Gamma,\Z(P)_{sa})$ and $H^2_b(\Gamma,L^2(\Z(P))_{sa})$ respectively.

When a unitary-valued $2$-cocycle $\omega$ has a uniform spectral gap, then we can take a logarithm to obtain a bounded $2$-cocycle taking values in (the additive module) $\Z(P)_{sa}$.  The easy lemma below collects this fact for later use.
\begin{lemma}\label{LogLem}
Suppose that $\alpha:\Gamma\curvearrowright P$ is a trace preserving action of a
discrete group on a finite von Neumann algebra $P$ with a separable predual.
Let $\omega\in Z^2(\Gamma,\U(\Z(P)))$ satisfy
$\sup_{g,h\in\Gamma}\|\omega(g,h)-I_P\|<\sqrt{2}$. 
\begin{enumerate}[(i)]
\item\label{LogLem:Part1} The expression $\psi=-i\log\omega$ defines a bounded cocycle in $Z^2_b(\Gamma,\Z(P)_{sa})$, i.e. \begin{equation}
\sup_{g,h\in\Gamma}\|\psi(g,h)\|<\infty,\quad
\alpha_g(\psi(h,k))-\psi(gh,k)+\psi(g,hk)-\psi(g,h)=0,\quad g,h,k\in \Gamma.
\end{equation}
\item\label{LogLem:Part2} If $\psi=\partial \phi$ for some $\phi\in C^1(\Gamma,\Z(P)_{sa})$, then $\omega=\partial\nu$, where $\nu(g)=e^{i\phi(g)}$ is a $1$-cochain in $C^1(\Gamma,\U(\Z(P)))$. 
\item\label{LogLem:Part3} Regarding $\psi$ as a cocycle taking values in $L^2(\Z(P))_{sa}$, if $\psi=\partial\phi$ for some $\phi\in C^1(\Gamma,L^2(\Z(P))_{sa})$, then $\omega=\partial\nu$ for $\nu(g)=e^{i\phi(g)}$.
\end{enumerate}
\end{lemma}

 We are grateful to Nicholas Monod for explaining to
us how the results of \cite{BM:GAFA,MS:JDG,M:Crelle} can be
combined to show that $H^2_b(\Gamma,\Z(P)_{sa})=0$ for suitable higher rank lattices $\Gamma$. The presentation  below is based on
\cite{M:Email}.  

While the module $\Z(P)_{sa}$ is not necessarily separable as a Banach space, it is a \emph{semiseparable} coefficient module in the sense of \cite[Definition 3.11]{M:Crelle}, as the canonical embedding $\Z(P)_{sa}\hookrightarrow L^2(\Z(P))_{sa}$ is the contragredient of $L^2(\Z(P)_{sa})\rightarrow L^1(\Z(P)_{sa})$. We have a direct sum decomposition of $\Z(P)_{sa}=\mathbb RI_P\oplus \Z(P)_{sa}^0$ as semi-separable coefficient modules
(where $\Z(P)^0_{sa}=\{z\in\Z(P)_{sa}:\tau_P(z)=0\}$ and $\Gamma$ acts trivially on
$\mathbb RI_P$). Thus $H^2_b(\Gamma,\Z(P)_{sa})=0$ if and only if $H^2_b(\Gamma,\Z(P)^0_{sa})=0$ and $H^2_b(\Gamma,\mathbb R)=0$.

Let $k$ be a local field (for example $\mathbb R$ or a nonarchimedian local field such as a finite extension of the $p$-adic numbers).  Let $G$ be a connected, almost $k$-simple algebraic group over $k$ with rank at least $2$ and let $\Gamma$ be a lattice in $G$. Burger and Monod show in \cite{BM:JEMS,BM:GAFA} (see the proof of \cite[Corollary 24]{BM:GAFA})
that $H^2_b(\Gamma,\mathbb R)$ is isomorphic to $H^2_c(G,\mathbb R)$, the
continuous cohomology of the underlying group $G$ (in which the cochain complex
consists of jointly continuous maps). This last group is known, and vanishes
unless $k=\mathbb R$ and $\pi_1(G)$ is infinite. In particular, $H^2_b(\Gamma,\mathbb R)=0$ when $\Gamma=SL_n(\mathbb
Z)$ with $n\geq 3$.  This vanishing result can be found
explicitly as \cite[Theorem 1.4]{MS:JDG} which shows further that, under the
same hypotheses on $\Gamma$, $H^2_b(\Gamma,E)=0$ for all separable coefficient
modules $E$. In particular, this shows that when $\alpha:\Gamma\curvearrowright
P$ is a properly outer, centrally ergodic, trace preserving action of such a
group $\Gamma$ on a finite von Neumann algebra $P$ with separable predual and
finite dimensional center, then $H^2_b(\Gamma,\Z(P)_{sa})=0$. 

In order to obtain vanishing results when $\Z(P)_{sa}$ is infinite dimensional,
and so nonseparable as a Banach space, we need to use the results of
\cite{M:Crelle} which handle semiseparable modules.  Let $G=\prod G_i(k_i)$ be
a finite product of connected, simply connected semisimple $k_i$-groups for
local fields $k_i$.  Then, given a lattice $\Gamma$ in $G$ and a semiseparable
coefficient module $V$ for $\Gamma$ with no invariant vectors, the bounded
cohomology groups $H^2_b(\Gamma,V)$ vanish provided the minimal rank of each
$k_i$-almost simple factor of $G_i$ is at least $2$ for every $i$ (see
\cite[Corollary 1.8]{M:Crelle}, and \cite[Corollary 1.6]{M:Crelle} for the
special case of a lattice in a connected, simply connected, almost simple
group).  In particular, we can take $V=\Z(P)_{sa}^0$ in this result for any
centrally ergodic, properly outer trace preserving action
$\alpha:\Gamma\curvearrowright P$ on a finite von Neumann algebra with separable
predual.

In the case of $SL_n(\mathbb Z)$, the previous two paragraphs give the
following result.
\begin{theorem}[Monod]\label{Monod}
Let $\Gamma=SL_n(\mathbb Z)$ for $n\geq 3$. Then, for any properly outer
centrally ergodic trace preserving action of $\Gamma$ on a finite von Neumann
algebra $P$ with separable predual, the cohomology group
$H^2_b(\Gamma,\Z(P)_{sa})$ vanishes.
\end{theorem}

In the case of irreducible lattices $\Gamma$ in finite products $G=\prod G_i(k_i)$ of at least $2$ factors as above, \cite[Corollary 24]{BM:GAFA} shows that $H^2_b(\Gamma,\mathbb R)=0$ when $G$ has no hermitian factors. Just as above, this can be combined with the results of \cite{M:Crelle} to show that the required bounded cohomology groups vanish.   For example, for $n\geq 3$ and a prime $p$, $SL_n(\mathbb Z[1/p])$ is an irreducible lattice in $SL_3(\mathbb R)\times SL_3(\mathbb Q_p)$, and so $H^2_b(SL_n(\mathbb Z[1/p]),\Z(P)_{sa})=0$ for all properly outer, centrally ergodic, trace preserving actions $SL_n(\mathbb Z[1/p])\curvearrowright P$ on a finite von Neumann algebra with separable predual.

Finally, we are grateful to one of the referees for bringing the following proposition to our attention, which enables us to achieve explicit constants in Theorem \ref{TA}.  The proof below uses the homogeneous complex $\cdots\rightarrow\ell^\infty(\Gamma^{n},X)^\Gamma\stackrel{\partial^n}{\rightarrow}\ell^\infty(\Gamma^{n+1},X)^\Gamma\rightarrow\cdots$ of $\Gamma$-invariant functions to define $H^\bullet_b(\Gamma,X)$; this is isometrically isomorphic to the inhomogeous complex used elsewhere in the paper.
\begin{proposition}\label{K=4}
Let $\Gamma$ be a countable discrete group and $X$ a semiseparable coefficient module.  Given $\psi\in Z^2_b(\Gamma,X)$, which represents the trivial class in $H^2_b(\Gamma,X)$, there exists $\phi\in C^1_b(\Gamma,X)$ with $\partial\phi=\psi$ and $\|\phi\|\leq 6\|\psi\|$.
\end{proposition}
\begin{proof}
The bounded cohomology groups $H^\bullet_b(\Gamma,X)$ can be computed using the Poisson boundary corresponding to a probability measure on $\Gamma$ from \cite{F:AdvP}. This is a strong boundary in the sense of \cite[Definition 2.3]{MS:JDG} (this is \cite{K:GAFA} for the double ergodicity condition and \cite{Z:JFA} for the amenability).  In particular, double ergodicity ensures that the module $L^\infty_{w^*,\alt}(B^2,X)^\Gamma$ of $\Gamma$-invariant alternating essentially bounded weak$^*$-measurable functions from $B^2$ into $X$ vanishes (see \cite[Lemma 5.5]{M:Crelle} for the easy extension to semi-separable modules).

Using the results in \cite[Section 7]{M:Book}, there exist contractive $\Gamma$-equivariant maps $\theta^n:\ell^\infty(\Gamma^{n+1},X)\rightarrow\ell^\infty(\Gamma^{n+1},X)$ which factor through the module $L^\infty_{w^*,\alt}(B^n,X)$ such that the diagram
\begin{equation}\label{K=4.1}
\xymatrix{0\ar[r]&X\ar[r]^{\eps\quad}\ar[d]_{\id_X}&{\ell^\infty(\Gamma,X)}\ar[d]_{\theta^0}\ar[r]^{\partial^1}&{\ell^\infty(\Gamma^2,X)}\ar[r]^{\partial^2}\ar[d]_{\theta^1}&{\ell^\infty(\Gamma^3,X)}\ar[d]_{\theta^2}\ar[r]^{\quad\partial^3}&{\cdots}\\0\ar[r]&X\ar[r]_{\eps\quad}&{\ell^\infty(\Gamma,X)}\ar[r]_{\partial^1}&{\ell^\infty(\Gamma^2,X)}\ar[r]_{\partial^2}&{\ell^\infty(\Gamma^3,X)}\ar[r]_{\quad\partial^3}&{\cdots}}
\end{equation}
commutes ($\eps$ denotes the map including $X$ as constant sequences).  First extend $\id_X$ to a $\Gamma$-morphism using \cite[Lemma 7.5.6]{M:Book} from the first row of (\ref{K=4.1}) to the augmented resolution
\begin{equation}\label{K=4.2}
0\rightarrow X\stackrel{\eps}{\rightarrow}L^\infty_{w^*}(B,X)\rightarrow L^\infty_{w^*}(B^2,X)\rightarrow\cdots
\end{equation}
(the second paragraph of the proof of \cite[Theorem 7.5.3]{M:Book}, shows how the amenability of the boundary provides the necessary hypothesis to use \cite[Lemma 7.5.6]{M:Book}).   One obtains $\theta^\bullet$ by following this morphism by the canonical projection $L^\infty_{w^*}(B^\bullet,X)\rightarrow L^\infty_{w^*,\alt}(B^m,X)$, and the inclusion $L^\infty_{w^*,\alt}(B^\bullet,X)\subseteq L^\infty_{w^*}(B^\bullet,X)\subseteq \ell^\infty(\Gamma^\bullet,X)$ arising from the Poisson transform. 

The maps $\id^\bullet$ and $\theta^\bullet$ are homotopic by \cite[Lemma 7.2.6]{M:Book}, so there exist $\Gamma$-equivariant bounded linear maps $\sigma^n:\ell^\infty(\Gamma^{n+1},X)\rightarrow\ell^\infty(\Gamma^n,X)$ such that 
\begin{equation}\label{K=4.3}
\id_{\ell^\infty(\Gamma^{n+1},X)}-\theta^n=\partial^n\sigma^n+\sigma^{n+1}\partial^{n+1},\quad n\geq 0.
\end{equation}  The proof of \cite[Lemma 7.2.6]{M:Book} is inductive and can be used to obtain estimates on $\|\sigma^n\|$. One takes $\sigma^0=0$ and then has $\|\sigma^{n+1}\|\leq\|\id_{\ell^\infty(\Gamma^{2},X)}-\theta^n-\partial^n\sigma^n\|$ for all $n\geq 1$ so that $\|\sigma^1\|\leq 2$ and $\|\sigma^2\|\leq \|\id_{\ell^\infty(\Gamma^{n+1},X)}\|+\|\theta^1\|+2\|\sigma^1\|\leq 6$.

Then, given $\psi\in \ell^\infty(\Gamma^3,X)^\Gamma$ representing a trivial class in $H^2_b(\Gamma,X)$, write $\psi=\partial^2(\tilde{\phi})$ for some $\tilde{\phi}\in\ell^\infty(\Gamma^2,X)^\Gamma$. As $\theta^1(\tilde{\phi})=0$, so $\theta^2(\psi)=0$ as the third square in (\ref{K=4.1}) commutes. Then (\ref{K=4.3}) gives $\psi=\partial^2\sigma^2(\psi)$, and so we can take $\phi=\sigma^2(\psi)$.
\end{proof}

\section{Normalizers of amenable subalgebras}\label{NormSection}

Normalizers were introduced by Dixmier \cite{D:Ann} in order to distinguish maximal abelian
subalgebras (masas) up to automorphisms of the larger algebra. We will use them to study the structure of twisted crossed products.
\begin{definition}
For an inclusion $P\subseteq M$ of von Neumann algebras, the group $\mathcal N(P\subseteq M)$ of
\emph{normalizers} of $P$ in $M$ consists of those unitaries $u\in M$ with
$uPu^*=P$.    The subalgebra $P$ is called  \emph{singular} 
if $\N(P\subseteq M)=\U(P)$ and \emph{regular} if $\N(P\subseteq M)''=M$.  When $P$ is a regular maximal abelian subalgebra of a II$_1$ factor, it is called \emph{Cartan}.
\end{definition}
Twisted crossed products provide the prototype of regular inclusions (the
algebra $P$ is always regular in $P\rtimes_{\alpha,\omega}\Gamma$). Given a II$_1$ factor $M$ and a regular von Neumann subalgebra $P$
containing $I_M$, a \emph{bounded homogenous orthonormal basis of normalizers}
(see \cite[Definition 4.1]{IoPeP:Acta}) for $P\subseteq M$ is a family
$(u_n)_{n\geq
0}$ in $\N(P\subseteq M)$ such that $u_0=I_M$,
$E_P^M(u_i^*u_j)=\delta_{i,j}I_M$ for all $i,j\geq 0$ and
$\sum_{n=0}^\infty u_nP$ is dense in $L^2(M)$. Note that the condition
$E_P^M(u_i^*u_j)=0$ for $i\neq j$ is equivalent to
$\overline{u_iP}\perp\overline{u_jP}$ in $L^2(M)$. Again, the prototypical
behavior is found in the twisted crossed product construction --- the
canonical unitaries $(u_g)_{g\in G}$ provide a bounded
homogeneous orthornormal basis of normalizers. More generally, such a basis can
always be found when  $P$ is a Cartan masa in $M$ \cite[Lemma
4.2]{IoPeP:Acta}.

We record the following proposition regarding close normalizers from
\cite{KR:CMP} in a form suitable for later use.  
\begin{proposition}\label{InjectiveUnitary}
Let $P\subseteq\mathcal B(\Hs)$ be a von Neumann algebra and suppose that
$u_1,u_2\in\N(P\subseteq \mathcal B(\Hs))$ satisfy $\|u_1-u_2\|<1$. Then there
exist unitaries $v\in P$ and $v'\in P'$ with $u_2=u_1vv'$,
$\|v-I_\Hs\|\leq\sqrt{2}\|u_1-u_2\|$, and $\|v'-I_\Hs\|\leq(\sqrt{2}+1)\|u_1-u_2\|$.
\end{proposition}
\begin{proof}
The automorphism $\theta=\Ad(u_1^*u_2)$ of $P$ has $\|\theta-\id_P\|\leq
2\|u_1-u_2\|<2$.   By \cite[Lemma 5]{KR:CMP}, there is a unitary $v\in P$ with
$\Ad(v)=\theta$ whose spectrum is contained in the half plane 
$\{z\in\mathbb C:\Re(z)\geq 2^{-1}\left(4-\|\theta-\id_P\|^2\right)^{1/2}\}$.
A routine computation  via the spectral radius gives
$
\|v-I_P\|\leq \sqrt{2}\|u_1-u_2\|
$.
Since $v^*u_1^*u_2\in P'$, we can write $u_2=u_1vv'$ for a unitary $v'\in P'$ with $\|v'-I_{P'}\|\leq \|v-I_P\|+\|u_1-u_2\|\leq (1+\sqrt{2})\|u_1-u_2\|$.
\end{proof}

Our main objective in this section (Lemma \ref{NormaliserLemma0}) is to exploit existing pertubation results from \cite{C:Acta,C:IJM} to examine the normalizer
structure of pairs of inclusions $P\subseteq M$ and $Q\subseteq N$ with $d(P,Q)$ and $d(M,N)$ small. When $P$ (and hence $Q$) is amenable or $M$ (and hence $N$) is finite we will
transfer normalizers of $P$ in $M$ to
normalizers of $Q$ in $N$.    We first record the original embedding and perturbation results for amenable von Neumann algebras below; they will be used both here, and also repeatedly henceforth.  In part (\ref{Injective:Part3}) below, note that the hypothesis of the original statement in \cite[Theorem 4.1]{C:IJM} is that $d(M,N)<1/8$, but the proof only needs the hypothesis $M\subset_\gamma N$ and $N\subset_\gamma M$ given here.
\begin{theorem}\label{Injective}
\begin{enumerate}[(i)]
\item\label{Injective:Part1} ({\cite[Theorem 4.3, Corollary 4.4]{C:Acta}})
Let $P$ be an amenable von Neumann subalgebra of $\mathcal B(\Hs)$ containing
$I_\Hs$. Suppose that
$B$ is another von Neumann subalgebra of $\mathcal B(\Hs)$ and $P\subset_\gamma
B$ for a constant $\gamma<1/100$.  Then there exists a unitary $u\in (P\cup
B)''$ with $uPu^*\subseteq B$, $\|I_\Hs-u\|\leq150\gamma$ and
$\|uxu^*-x\|\leq100\gamma\|x\|$
for $x\in P$. If, in addition, $\gamma<1/101$ and $B\subset_\gamma P$, then
$uPu^*=B$.
\item\label{Injective:Part2} ({\cite[Corollary 4.2]{C:Acta}})
If $M,N\subseteq\mathcal B(\Hs)$ are amenable von Neumann algebras containing $I_\Hs$ and $M\subset_{\gamma}N$ for a
constant $\gamma<1/8$ then
there is a unitary $u\in (M\cup N)''$ such that $\|I_\Hs-u\|\leq 12\gamma$ and
$uMu^*\subseteq N$. Additionally,
if $N\subset_{\gamma}M$ then $u$ may be chosen to also satisfy $uMu^*=N$.
\item\label{Injective:Part3} (\cite[Theorem 4.1]{C:IJM})
If $P$ and $Q$ are unital von Neumann subalgebras of a finite von Neumann
algebra $M$ satisfying
$P\subset_{\gamma}Q\subset_{\gamma}P$ for a constant $\gamma<1/8$ then there is
a unitary $u\in (P\cup
Q)''\subseteq M$ with $\|I_M-u\|\leq 7\gamma$ such that $uPu^*=Q$.
\end{enumerate}
\end{theorem}

\begin{lemma}\label{NormaliserLemma0}
Let $M$ and $N$ be von Neumann subalgebras of $\mathcal{B}(\Hs)$ containing
$I_{\Hs}$
and satisfying
$M\subset_{\gamma}N\subset_{\gamma}M$ for a constant $\gamma>0$. Let
$P\subseteq M$ and $Q\subseteq N$ be
von Neumann subalgebras satisfying $P\subset_{\delta}Q\subset_{\delta}P$ for a
constant $\delta \geq 0$.
\begin{enumerate}[(i)]
\item\label{NormaliserLemma0:Item1} Suppose that $P$ and $Q$ are
amenable and that $2\delta+\delta^2+2\sqrt{2}\gamma <1/8$. Given $v\in
\N(P\subseteq M)$, there
exists $v'\in \N(Q\subseteq N)$ with $\|v-v'\|< 25 \delta +25\sqrt{2}\gamma$.
\item\label{NormaliserLemma0:Item2} Suppose that  $M$ and $N$ are finite von
Neumann algebras, and that $2\delta+\delta^2+2\sqrt{2}\gamma <1/8$. Then the
conclusion of (\ref{NormaliserLemma0:Item1}) holds with the improved estimate 
$\|v-v'\|<
15\delta+15\sqrt{2}\gamma$.
\item\label{NormaliserLemma0:Item3}
Suppose that $P$ is amenable, that $\gamma<1/(2\sqrt{2})$ and that $Q=P$. Given
$v\in \N(P\subseteq
M)$, there exists $v'\in
\N(P\subseteq N)$ such that $\|v-v'\|\leq (4+2\sqrt{2})\gamma$.
\item\label{NormaliserLemma0:Item5} Suppose that $P$ and $Q$ are amenable.
Given $v\in\N(P\subseteq M)$ and
$v'\in\N(Q\subseteq N)$ with
$\|v-v'\|<1-24\delta$,   there exist unitaries $u\in (P\cup Q)''$, $w\in P$
and $w'\in P'$ satisfying $\|u-I_\Hs\|\leq 12\delta$,
 $\|w-I_\Hs\|<2^{1/2}(24\delta+\|v-v'\|)$, $\|w'-I_\Hs\|<
(2^{1/2}+1)(24\delta+\|v-v'\|)$, and $u^*v'u=vww'$.
\end{enumerate}
\end{lemma}
\begin{proof}
(\ref{NormaliserLemma0:Item1})-(\ref{NormaliserLemma0:Item2}).  Let $v\in \N(P\subseteq M)$ and choose a unitary $u\in N$ with
$\|u-v\|< \sqrt{2}\gamma$ by Lemma
\ref{ProjLem} (\ref{ProjLem:Part1}). Fix $x\in Q$ with $\|x\|\leq 1$, and choose $y\in P$ so that
$\|x-y\|<\delta$, in which case
$\|y\|<1+\delta$. Then $\|vxv^*-vyv^*\|<\delta$ and
$\|uxu^*-vxv^*\|<2\sqrt{2}\gamma$ so
$\|uxu^*-vyv^*\|<\delta+2\sqrt{2}\gamma$. Since $vyv^*\in P$, there exists
$x_1\in Q$ with
$\|vyv^*-x_1\|<\delta(1+\delta)$, and so
$\|uxu^*-x_1\|<2\delta +\delta^2+2\sqrt{2}\gamma$. This shows that
$uQu^*\subset_{2\delta+\delta^2+2\sqrt{2}\gamma}Q\subset_{
2\delta+\delta^2+2\sqrt{2}\gamma}uQu^*$, where the second near
inclusion follows by applying the
same argument to $v^*$ and $u^*$.

For (\ref{NormaliserLemma0:Item1}) we are assuming that $P$ and $Q$ are amenable, so from Theorem
\ref{Injective} (\ref{Injective:Part2}), there exists a unitary $w\in 
N$ so that $wuQu^*w^*=Q$ and $\|I_\Hs-w\|\leq 12(2\delta+\delta^2+2\sqrt{2}\gamma)$.
Define $v'=wu\in
\N(Q\subseteq N)$. Then, since $\delta<1/16$,
\begin{equation}
\|v-v'\|=\|v-wu\|
\leq \|v-u\|+\|I_\Hs-w\|
<25\delta+25\sqrt{2}\gamma.
\end{equation}

For (\ref{NormaliserLemma0:Item2}) we assume that  $M$ and $N$ are finite von Neumann algebras with no
restrictions on $P$ and $Q$. The counterpart of the unitary $w$ above
may now
be
chosen so that $\|I_\Hs-w\|\leq
14\delta +7\delta^2+14\sqrt{2}\gamma$ using Theorem \ref{Injective} (\ref{Injective:Part3}),
leading to
the estimate
$\|v-v'\|<15\delta+15\sqrt{2}\gamma$. 

(\ref{NormaliserLemma0:Item3}). Now suppose that $P$ is amenable and $P=Q$. If $v\in
\N(P\subseteq M)$,
then Lemma \ref{ProjLem} (\ref{ProjLem:Part1}) allows us to choose a
unitary $u\in N$ with
$\|v-u\|\leq \sqrt{2}\gamma$. Thus $x\mapsto u^*vxv^*u$ defines an isomorphism
$\phi$ of $P$ into $N$ with
$\|\phi(x)-x\|\leq 2\|u-v\|\leq 2\sqrt{2}\gamma < 1$ for $x\in P$, $\|x\|\leq
1$. From \cite[Theorem
4.2]{C:Invent}, there exists a unitary
$w\in N$ so that $\phi(x)=wxw^*$
and $\|w-I_\Hs\|\leq 4\gamma$. Define $v'=uw\in \N(P\subseteq N)$. The required estimate follows from
\begin{equation}
\|v-v'\|=\|v-uw\|=\|w-u^*v\|
\leq \|w-I_\Hs\|+\|I_\Hs-u^*v\|
\leq (4+\sqrt{2})\gamma.
\end{equation}
(\ref{NormaliserLemma0:Item5}). The hypothesis carries an implicit assumption that 
$\delta<1/24$, so  Theorem \ref{Injective} (\ref{Injective:Part2}) allows us to choose a unitary
$u\in 
(P\cup Q)''$ with $\|u-I_\Hs\|\leq 12\delta$ and $uPu^*=Q$. Let $u_1=u^*v'u$ and
$u_2=v$ which both normalize $P$. Then, since $\|v-v'\|<1-24\delta$, 
\begin{equation}
\|u_1-u_2\|\leq \|u^*v'u-u^*vu\|+\|u^*vu-v\|
< 1-24\delta+2\|u-I_\Hs\|\leq 1-24\delta+24\delta<1.
\end{equation}
By Proposition \ref{InjectiveUnitary} there exist unitaries $w\in P$ and $w'\in
P'$ with $u_1=u_2ww'$, $\|w-I_\Hs\|\leq
\sqrt{2}\|u_1-u_2\|$ and $\|w'-I_\Hs\|\leq (\sqrt{2}+1)\|u_1-u_2\|$. Thus
$u^*v'u=vww'$ with $\|w-I_\Hs\|\leq
\sqrt{2}(\|v-v'\|+24\delta)$ and $\|w'-I_\Hs\|\leq
(\sqrt{2}+1)(\|v-v'\|+24\delta)$.
\end{proof}

\section{Reduction to standard position}\label{Cartan}

Given two close   II$_1$ factors $M$ and $N$ on a
general Hilbert space $\Hs$, our main objective in this section is to show how we can
construct close isomorphic copies of these algebras on a new Hilbert space $\Ks$
so that they both act in standard position on $\Ks$.  Given an amenable subalgebra $P\subseteq M\cap N$ with $P'\cap M\subseteq P$, by working with both algebras in standard position we can show that $P$ is regular in $M$ if and only if it is regular in $N$ (see Lemma 
\ref{generate}).

Difficulties arise because  the usual strategy of changing representations by
amplification and compression must be employed in a fashion compatible with both
algebras. In particular, as we do not assume that $M'$ and $N'$ are close on
$\Hs$, we must take care in compressing by projections in $M'$; these need not
be close to projections in $N'$. Matters are much simpler when $M'$ and $N'$ are assumed to be close.

The first step is the following lemma, which improves (in the factor case for simplicity) the technical result \cite[Lemma 3.7]{CSSW:GAFA} by removing the restriction that the commutants $M'$ and $N'$ are close.

\begin{lemma}\label{Std.Lem1}
Let $M$ and $N$ be   {\rm{II}}$_1$ factors acting %non-degenerately 
on $\Hs$ with $\dim_M(\Hs)\leq 1$. If
 $M\subset_\gamma N$ and $N\subset_\gamma M$ for some
$\gamma >0$, then the following  hold:
\begin{enumerate}[(i)]
\item\label{Std.Lem1:Part1} $N'\subset_{2{(1+\sqrt{2})}\gamma}M'$;
\item\label{Std.Lem1:Part2} if $\gamma<1/22$, then 
  $N'$ is finite;
\item\label{Std.Lem1:Part3} if $\gamma<1/47$, then 
$M'\subset_{4{(1+\sqrt{2})}\gamma}
N'$.
\end{enumerate}
\end{lemma}
\begin{proof} (\ref{Std.Lem1:Part1}). Since $M$ has a cyclic vector on $\Hs$, this is Proposition \ref{LocalDK}~(\ref{LocalDK:Part2}).

 (\ref{Std.Lem1:Part2}).  Suppose that
 $\gamma<1/22$  and  that $N'$ is infinite. Take $v\in N'$
with $v^*v=I_\Hs$ and
$vv^*=e$ for some projection $e\neq I_\Hs$. Then, by part (i), there is an element
$w\in M'$
with $\|w-v\|\leq 2{(1+\sqrt{2})}\gamma$, which implies that $\|w\|\leq
1+2{(1+\sqrt{2})}\gamma$. Since
$v^*v=I_\Hs$, it follows that
\begin{align}
\|w^*w-I_\Hs\|&=\|w^*w-v^*v\|
\leq \|w^*w-w^*v\|+\|w^*v-v^*v\|\notag\\
&\leq (1+2{(1+\sqrt{2})}\gamma)2{(1+\sqrt{2})}\gamma
+2{(1+\sqrt{2})}\gamma
=(2+2{(1+\sqrt{2})}\gamma)(2{(1+\sqrt{2})}\gamma),
\end{align}
and similarly
$\|ww^*-vv^*\|\leq (2+2{(1+\sqrt{2})}\gamma)(2{(1+\sqrt{2})}\gamma)$. As
$\gamma<1/22$,  it follows that $\|w^*w-I_\Hs\|<1/2<1$ and so
 $w^*w$ is invertible in $M'$. Since $M'$ is  II$_1$, we can write
$w=u|w|$, where $u$ is a unitary.  Then $ww^*=uw^*wu^*$, so
 $\|ww^*-I_\Hs\|=\|u(w^*w-I_\Hs)u^*\|<1/2$. Thus $
\|e-I_\Hs\|\leq \|vv^*-ww^*\|+\|ww^*-I_\Hs\|<\frac{1}{2}+\frac{1}{2}=1$, a contradiction. 

(\ref{Std.Lem1:Part3}). 
Assume now that $\gamma< 1/47$. Since $\dim_{M'}(\Hs)\geq 1$, there is a tracial vector
$\eta\in \Hs$ for $M'$. Now by (i), $N'\subset_{2{(1+\sqrt{2})}\gamma}M'$ where
$2{(1+\sqrt{2})}\gamma<1$, and
$N'$ is finite by (ii). Apply Lemma \ref{Trace} to this pair to obtain
$|\tau_{N'}(y)-\langle y\eta,\eta\rangle
|<4{(1+\sqrt{2})}^2\gamma\|y\|$ for $y\in N'$.
Let $J=\{y\in N':y\eta=0\}$. Then $J$ is a weakly closed left ideal in $N'$ so
has the form $N'p$ for a
projection $p\in N'$. Since $p\eta=0$, we obtain
$\tau_{N'}(p)< 4(1+\sqrt{2})^2\gamma<1/2$, as $\gamma<1/47$.  Then $\tau_{N'}(I_\Hs-p)>1/2$, so $p\sim q\leq I_\Hs-p$  via $u\in \U(N')$ with 
  $upu^*=q$. Define $\zeta=u\eta$. 

Now suppose that $x\in N'$ satisfies $x\eta=x\zeta=0$. Then $x\in J$ so $xp=x$
and $x(I_\Hs-p)=0$. Since
$x\zeta=0$, we have $xu\eta=0$ so $xu\in J$ and $xu=xup$. Thus, since $q\leq
I_\Hs-p$,
\begin{equation}
x=xupu^*=xq=x(I_\Hs-p)q=0.
\end{equation}
This proves that the pair $\{\eta,\zeta\}$ is a separating set for $N'$, and so
$\{\eta,\zeta\}$ is a cyclic
set for $N$. From Proposition \ref{LocalDK} (\ref{LocalDK:Part2}), it follows that
$M'\subset_{4{(1+\sqrt{2})}\gamma}N'$ as required.\end{proof}

Our next aim is to show that if a II$_1$ factor $M$ acts
in standard position on $\Hs$, then the same is true for any close algebra $N$.
We initially work in the context of a distinguished subalgebra $A$ which is a masa in $M$
and in $N$ and satisfies $J_MAJ_M\subseteq N'$. In this case $M$, $M'$, $N$ and
$N'$
can be taken to have exactly the same tracial vector and so we can compare the
basic construction algebras  obtained from subalgebras of $M\cap N$;
 the inclusions $A\subseteq M$ and
$A\subseteq N$ have the same basic construction algebra $\langle
M,e_A\rangle=\langle N,e_A\rangle$. 

\begin{lemma}\label{Std.Lem2}
Let $M$ be a  {\rm{II}}$_1$ factor  acting in standard
position on a Hilbert space $\Hs$ with tracial vector $\xi$, and let $N$ be 
another    {\rm{II}}$_1$ factor on
$\Hs$. Suppose that
$M\subset_\gamma N$ and $N\subset_\gamma M$ for some constant
$\gamma<1/47$. Suppose further that $A$ is a masa in
both $M$
and $N$ and
$J_MAJ_M\subseteq N'$. Then the following hold:
\begin{enumerate}[(i)]
\item\label{Std.Lem2:Part1}$\xi$ is a tracial vector for both $N$ and $N'$ so that $N$
is
also in standard position on $\Hs$.
\item\label{Std.Lem2:Part2}$(M\cup \{e_A\})''=(J_MAJ_M)'=(J_NAJ_N)'=(N\cup \{e_A\})''$,
where $J_N$ is the modular conjugation operator induced from the cyclic and
separating vector $\xi$ for $N$ and $e_A$ is the projection onto $\overline{A\xi}$.
\end{enumerate}
\end{lemma}
\begin{proof}
(\ref{Std.Lem2:Part1}). For each unitary $u\in A$ and $y\in N$, we have
\begin{equation}\label{Std.Lem2.Eq1}
\langle uyu^*\xi,\xi\rangle=\langle
yu^*\xi,u^*\xi\rangle=\langle
yJ_MuJ_M\xi,J_MuJ_M\xi\rangle=\langle y\xi,\xi\rangle,
\end{equation}
as $J_MuJ_M\in J_MAJ_M\subseteq N'$.  Since $A'\cap N=A$, the unique
$\tau_N$-preserving
conditional expectation $E^N_A:N\rightarrow A$ is given by taking
$E^N_A(y)$ to
be the unique element of minimal norm in  $\overline{\mathrm{co}}^{\|\cdot\|_{2,\tau_N}}\{uyu^*:u\in\mathcal U(A)\}$ 
\cite[Lemma 3.6.5]{SS:Book2}. Then
(\ref{Std.Lem2.Eq1}) gives
$\langle y\xi,\xi\rangle=\langle
E_A^N(y)\xi,\xi\rangle=\tau_M(E_A^N(y))$ 
for $y\in N$. As $\gamma<2^{-3/2}$, applying Lemma \ref{MasaTr} 
gives $\tau_M|_A=\tau_N|_A$. It follows that
$\tau_N(y)=\tau_N(E_A^N(y))=\tau_M(E_A^N(y))=\langle
y\xi,\xi\rangle$ for $y\in N,$
so  $\xi$ is tracial for $N$.  

As $\gamma<1/47$,  Lemma
\ref{Std.Lem1} shows that $N'$ is finite and that $M'\subset_{4{(1+\sqrt{2})}\gamma} N'$ and
$N'\subset_{2{(1+\sqrt{2})}\gamma}M'$. 
 Since $4{(1+\sqrt{2})}\gamma<1$, Lemma \ref{relcom2} shows that
$J_MAJ_M$ is a masa of $N'$
as $J_MAJ_M$ is a masa of $M'$. Since
 $4{(1+\sqrt{2})}\gamma<2^{-3/2}$, applying the previous paragraph to the pair
$(M',N')$ shows that $\xi$ is tracial for $N'$, establishing (\ref{Std.Lem2:Part1}).

(\ref{Std.Lem2:Part2}). Since both $M$ and $N$ are in standard position with respect
to the same
cyclic and separating vector $\xi$ (though with possibly different modular
conjugation operators $J_M$ and $J_N$), the Hilbert space projection $e_A$ from
$\Hs$ onto $\overline{A\xi}$ used to define the basic construction
$\langle M,e_A\rangle$ from the inclusion $A\subseteq M$ is the same
projection used
to define the basic construction from the inclusion $A\subseteq N$.  Now 
\begin{equation}\label{Std.Lem2.Eq4}
J_MAJ_M\subseteq N'\cap \{e_A\}'=(J_NNJ_N)\cap \{e_A\}'=J_NAJ_N\subseteq
J_NNJ_N,
\end{equation}
where we use $e_A=J_Ne_AJ_N=J_Me_AJ_M$ and
$A=N\cap\{e_A\}'=M\cap\{e_A\}'$ from Properties \ref{bcon-prop} (\ref{bcon-prop:Item1}) and (\ref{bcon-prop:Item2}). 
As we noted in the previous paragraph,
$J_MAJ_M$ is a maximal abelian subalgebra of $N'=J_NNJ_N$, giving the equality
$J_MAJ_M=J_NAJ_N$.  Taking commutants gives the middle inclusion
$(J_MAJ_M)'=(J_NAJ_N)'$ of (ii).  The two outermost equalities $(M\cup
\{e_A\})''=(J_MAJ_M)'$ and $(J_NAJ_N)'=(N\cup \{e_A\})''$ are the
characterization of the basic construction algebra given in \cite[Proposition
3.1.5 (i)]{J:Invent} (see Properties \ref{bcon-prop} (\ref{bcon-prop:Item3})).
\end{proof}

Next, we replace the masa $A$ in the previous result by an amenable
von Neumann subalgebra $P\subseteq M$ with $P'\cap M\subseteq P$ satisfying
$J_MPJ_M\subseteq N'$ and show that the basic constructions $\langle
M,e_P\rangle$ and $\langle N,e_P\rangle$ are equal. It follows that $P$ is regular in $M$ if and only if it is regular in $N$.  To do this we use
the following theorem of Popa (\cite[Theorem 3.2]{P:Invent}), which we quote from \cite[Theorem 12.2.4]{SS:Book2} since this records
additional information which is implicit in \cite{P:Invent}.
\begin{theorem}[Popa]\label{PopaThm}
 Let $P$ be a von Neumann subalgebra of a finite
von Neumann algebra $M$ with separable predual and suppose that $P'\cap
M\subseteq P$. If $A_0$ is a finite dimensional abelian *-subalgebra of $P$,
then there exists a masa $A$ in $M$ such that $A_0\subseteq A\subseteq P$.
\end{theorem}
 
\begin{lemma}\label{generate}
Let $M$ be a   ${\mathrm{II}}_1$ factor with separable predual acting in
standard position on a Hilbert space $\Hs$ with tracial vector $\xi$. Let $N$ be
a   ${\mathrm{II}}_1$ factor on $\Hs$ with $M\subset_\gamma N$ and
$N\subset_\gamma M$ for some $\gamma<1/47$. Suppose
that $P$ is
an
amenable subalgebra of $M\cap N$ with $P'\cap M\subseteq P$ and
$J_MPJ_M\subseteq N'$. Then the following statements hold.
\begin{enumerate}[(i)]
\item\label{generate:Part1} $N$ is also in standard position on $\Hs$ and $\xi$ is a tracial
vector for
$N$ and $N'$.
\item\label{generate:Part2}  $\langle M,e_P\rangle=(J_MPJ_M)'=(J_NPJ_N)'=\langle
N,e_P\rangle$, where $e_P$ is the projection onto $\overline{P\xi}$.
\item\label{generate:Part3} Suppose that $P\subseteq M$ is a regular inclusion with a bounded homogeneous basis of normalizers $(u_n)_{n\geq 0}$ and $(v_n)_{n\geq 0}$ is a family of normalizers in
$\N(P\subseteq N)$ with $v_0=I_\Hs$ and satisfying $\|v_n-u_n\|<1$ for all $n$.
Then
$(v_n)_{n\geq 0}$ is a bounded homogeneous basis of normalizers for $P\subseteq
N$ and so $(P\cup\{v_n:n\geq 0\})''=N$.
\item\label{generate:Part4}  $\N(P\subseteq M)$ generates $M$ if and only if $\N(P\subseteq
N)$ generates $N$. 
\end{enumerate}
\end{lemma}
\begin{proof}
(\ref{generate:Part1}). Since $P$ is amenable and satisfies $P'\cap M\subseteq P$,
Lemma \ref{relcom2} shows that $P'\cap N\subseteq P$. By Theorem \ref{PopaThm} there is a masa $A$ in $M$ with $A\subseteq P$. Then $J_MAJ_M\subseteq J_MPJ_M\subseteq N'$. By Lemma \ref{relcom2},
$A$ is also a masa in $N$. By Lemma \ref{Std.Lem2}, $N$
is in standard position on $\Hs$ with tracial vector $\xi$.

(\ref{generate:Part2}). Arguing just as in equation (\ref{Std.Lem2.Eq4}) in Lemma
\ref{Std.Lem2}, we have
\begin{equation}
J_MPJ_M\subseteq N'\cap\{e_P\}'=(J_NNJ_N)\cap\{e_P\}'=J_NPJ_N.
\end{equation}
For the reverse inclusion, given a finite dimensional
abelian subalgebra $A_0$ of $P$, use Theorem \ref{PopaThm} to find a masa $A_1$ in
$M$
with $A_0\subseteq A_1\subseteq P$. By Lemma \ref{relcom2}, $A_1$ is a masa
in $N$. Since $J_MA_1J_M\subseteq N'$, Lemma
\ref{Std.Lem2} (\ref{Std.Lem2:Part2}) gives $J_MA_1J_M=J_NA_1J_N$.  Thus $J_NA_0J_N\subseteq
J_MA_1J_M\subseteq J_MPJ_M$.  Fix a self-adjoint operator $x\in P$ and choose a
sequence $(x_n)_{n=1}^\infty$ of elements of $P$ which converge to $x$ in the weak operator
topology such that each $W^*(x_n)$ is a finite dimensional abelian von Neumann
algebra.  Applying the previous argument with $A_0=W^*(x_n)$ gives $J_Nx_nJ_N\in
J_MPJ_M$, and so taking weak operator limits, $J_NxJ_N\in J_MPJ_M$. This gives
$J_NPJ_N\subseteq J_MPJ_M$, and hence $J_NPJ_N=J_MPJ_M$.  The middle equality in
(\ref{generate:Part2}) follows by taking commutants, and the outer two equalities
by applying Properties \ref{bcon-prop} (\ref{bcon-prop:Item3}).

(\ref{generate:Part3}). For each $n$, apply Proposition
\ref{InjectiveUnitary}
to $u_n$ and $v_n$ to obtain unitaries $w_n\in P$ and $w_n'\in P'$ with
$u_n=v_nw_nw_n'$. As $N\subseteq \langle N,e_P\rangle=\langle M,e_P\rangle$, we
have $w_n'\in P'\cap \langle M,e_P\rangle$ and so, by Lemma
\ref{basiccon-lem}, there exist unitaries $z_n\in \Z(P)$ with
$w_n'\xi=z_n\xi$.  Thus $u_n\xi=v_nw_nz_n\xi$ and so
we have $\overline{u_nP\xi}=\overline{v_nP\xi}$.  As noted in Section
\ref{NormSection}, the condition $E_P^M(u_m^*u_n)=0$ for $m\neq n$ is
equivalent to $\overline{u_mP\xi}\perp\overline{u_nP\xi}$. Thus the
sequence $(v_n)_{n\geq 0}$ inherits this property and so satisfies
$E^N_P(v_m^*v_n)=\delta_{m,n}I$. Further, $\sum_{n=0}^\infty v_nP\xi$ is dense in $\Hs$ and so $(v_n)_{n=0}^\infty$ is a bounded homogeneous basis of normalizers for $P\subseteq N$. Note that this immediately implies that $(P\cup \{v_n:n\geq 0\})''=N$.  Indeed, if $N_0\subseteq N$ is the von Neumann algebra generated by $P$ and $(v_n)_{n\geq 0}$, consider $x\in N$ with $E^N_{N_0}(x)=0$. For each
$n\geq 0$ and $b\in
P$,
\begin{equation}\label{generate.7}
0=\tau_N(b^*v_n^*E
_{N_0}^N(x))=\tau_N(E_{N_0}^N(b^*v_n^*x))=\tau_N(b^*v_n^*x)=\langle x\xi
, v_nb\xi \rangle,
\end{equation}
since $\xi$ is a tracial vector for $N$. Since $\sum_{n=0}^\infty v_nP\xi$
is
dense in $\Hs$, (\ref{generate.7}) gives $x\xi=0$, and hence $x=0$ as
$\xi$ is separating for $N$.  Thus $N=N_0$ and $P\cup\{v_n:n\geq 0\}$
generates $N$.

(\ref{generate:Part4}). The hypotheses of Lemma \ref{NormaliserLemma0} (\ref{NormaliserLemma0:Item3}) are
satisfied  so given any normalizer $v\in\N(P\subseteq
M)$, there exists $u\in\N(P\subseteq N)$ with $\|v-u\|\leq\alpha$, where
$\alpha=(4+2\sqrt{2})\gamma$. Using $2\alpha<1$ and arguing just as in the first
paragraph of (\ref{generate:Part3}), we see that $\overline{uP\xi}=\overline{vP\xi}$.
   Suppose that $\N(P\subseteq M)$ generates $M$, so that
$\Hs=\overline{\rm{span}}\,\{u\xi:u\in\N(P\subseteq M)\}$. Thus
$\Hs=\overline{\rm{span}}\,\{v\xi:v\in\N(P\subseteq N)\}$.  Just as in the
second paragraph of (\ref{generate:Part3}), it then follows that $N=\N(P\subseteq N)''$.

For the reverse implication, we use parts (\ref{generate:Part1}) and (\ref{generate:Part2}) to interchange the roles of
$M$ and $N$. We have already noted that $P'\cap N\subseteq P$ and part (\ref{generate:Part2})
shows that $J_NPJ_N=J_MPJ_M\subseteq M'$. Thus if $\N(P\subseteq N)$ generates
$N$, then $\N(P\subseteq M)$ generates $M$.
\end{proof}

\begin{remark}\label{comrem}
In the special case that $P$ is a masa in $M$ in Lemma \ref{generate}, it
follows immediately that $P$ is Cartan in $M$ if and only if it is Cartan in
$N$. This can also be read off from Popa's characterization of
Cartan and singular masas in terms of the structure of the basic construction (\cite[Proposition 1.4.3(i)]{P:Ann}): a masa $B$ in a   II$_1$ factor $Q$
with a
separable predual is Cartan if and only if $B'\cap\langle Q,e_B\rangle$ is
generated by finite projections from $\langle Q,e_B\rangle$.  Thus, once we know
that the masa $P$ satisfies $\langle M,e_P\rangle=\langle N,e_P\rangle$, it
follows that it is Cartan in $M$ if and only if 
it is Cartan in $N$.
\end{remark}
When $M$ and $N$ are close II$_1$ factors on a Hilbert space $\Hs$ with a
cyclic vector for $M$, then the results so far show that $\dim_M(\Hs)=\dim_N(\Hs)$. In the statement below, we make a temporary assumption that $M$ and $N$ have a common masa.
\begin{proposition}\label{dim}
Suppose that $M$ and $N$ are   {\rm{II}}$_1$ factors
acting
nondegenerately on $\Hs$ with $M\subset_\gamma N$ and $N\subset_\gamma M$
for $\gamma<1/136209$. Moreover, suppose that $M\cap N$ contains an abelian von Neumann algebra $A$
which is a masa in $M$. If $\dim_M(\Hs)\leq 1$,
then
$\dim_N(\Hs)=\dim_M(\Hs)$.
\end{proposition}
\begin{proof} Note that $A$ is also a masa in $N$ by Lemma \ref{relcom2}.
Suppose first that $\dim_M(\Hs)=1$ so that $M$ is in standard position on $\Hs$
with a tracial vector $\xi$   for $M$ and $M'$. 
Since $M$ is in standard position, Lemma \ref{Std.Lem1} (iii) gives 
$J_MAJ_M\subseteq J_MMJ_M=M'\subset_{4{(1+\sqrt{2})}\gamma}N'.
$
The bound
on
$\gamma$ ensures that $4{(1+\sqrt{2})}\gamma<1/100$, so we can apply
Theorem \ref{Injective} (i)
to obtain a unitary
$v\in (J_MAJ_M\cup N')''\subseteq (J_MMJ_M\cup N')''=(M\cap N)'\subseteq
A'$ with $vJ_MAJ_Mv^*\subseteq N'$ and
\begin{equation}\label{dim.4}
\|v-I_\Hs\|\leq 600(1+\sqrt{2})\gamma,\quad
d(v(J_MAJ_M)v^*,J_MAJ_M)\leq400(1+\sqrt{2})\gamma.
\end{equation}
Define $N_1=v^*Nv$ so that $J_MAJ_M\subseteq N_1'$.
Then \eqref{NearTriangleUnitary} gives $M\subset_{\gamma_1}N_1$ and $N_1\subset_{\gamma_1} M$
where $\gamma_1:=(1200(1+\sqrt{2})+1)\gamma$. Since $A$ is a masa in $N$, it is also a masa in $N_1$. The initial bound on
$\gamma$ ensures that
$\gamma_1<1/47$, so we can apply Lemma \ref{Std.Lem2} (\ref{Std.Lem2:Part1})
to $N_1$
to
conclude that $N_1$ is in standard
position on $\Hs$. Since $N_1$ and $N$ are unitary conjugates, we also have
$\dim_N(\Hs)=1$.

Now suppose that $\dim_M(\Hs)<1$. Choose a projection $p\in A$ with
$\tau_M(p)=\dim_M(\Hs)$. We can cut by $p$ to obtain
$pMp\subset_{\gamma}pNp\subset_{\gamma}pMp$. Since $pMp$ is in standard position on $p\Hs$, the previous paragraph shows that so too is $pNp$. As
$\tau_{N}(p)=\tau_M(p)$ by Lemma \ref{MasaTr}, we see that
$\dim_{N}(\Hs)=\tau_{N}(p)=\tau_M(p)=\dim_M(\Hs)$ as required.
\end{proof}

Whether it is possible to drop the cyclic vector assumption in the previous result cuts to the heart of the (possible) difference between Kadison-Kastler stability and weak Kadison-Kastler stability for II$_1$ factors.
\begin{question}\label{QNew}
Does there exist a universal constant $\gamma_0>0$ such that whenever $M,N\subseteq \mathcal B(\Hs)$ are II$_1$ factors with $d(M,N)<\gamma_0$, then $\dim_M(\Hs)=\dim_N(\Hs)$?
\end{question}

In the absence of an answer to the previous question, the next lemma 
is designed to handle the situation when $\dim_\Hs(M)$ is large.
It  enables us to cut $M$ by a projection $e\in M'$ which almost lies in $N'$
such that $\dim_{Me}(e\Hs)\leq 1$, so that the previous results apply.

\begin{lemma}\label{Cyclic}
Let $M$ and $N$ be von Neumann algebras acting nondegenerately on a Hilbert
space $\Hs$
with $M\subset_\gamma N$ and $N\subset_\gamma M$ for a constant $\gamma>0$.
Given a unit vector $\zeta\in\Hs$,
 there exists a nonzero subprojection $e\in M'$ of
the projection with range $\overline{M\zeta}$ satisfying 
\begin{equation}\label{eqCyclic1}
{\mathrm{dist}}\,(e,N')\leq 6{(1+\sqrt{2})}\gamma
+2({(1+\sqrt{2})}\gamma)^{1/2},
\end{equation}
and if $M$ and $N$  are   ${\mathrm{II}}_1$ factors then $e$ may be chosen
with the additional property (T)hat $\dim_{Me}(e\Hs)=1/n$ for an integer $n$.
Moreover, if $\gamma$ satisfies $\gamma<1/87$
then there exists a projection $f\in N'$ and a unitary $u\in (M'\cup N')''$ so
that $ueu^*=f$ and
\begin{equation}\label{eqCyclic2}
\|e-f\|\leq 12{(1+\sqrt{2})}\gamma +4({(1+\sqrt{2})}\gamma)^{1/2},\ \
\|u-I_\Hs\|\leq
\sqrt{2}\|e-f\|.
\end{equation}
\end{lemma}
\begin{proof}
Choose $\gamma'<\gamma$ which satisfies $M\subseteq_{\gamma'}N$ and
$N\subseteq_{\gamma'}M$. Fix a unit vector $\zeta\in\Hs$, let  $p\in M'$ be
the projection onto $\overline{M\zeta}$ and let $q\in N'$ be the projection onto
$\overline{N\zeta}$.  Given $x$ in the unit ball of $M$, choose $y\in N$ with
$\|x-y\|\leq \gamma'$. Since $p$ commutes with $x$ and $q$ commutes with $y$, we
have the algebraic identity
$(px)pqp-pqp(px)=p(x-y)qp +pq(y-x)p$, leading to the estimate
$\|(px)pqp-pqp(px)\|\leq 2\|x-y\|\leq 2\gamma'$,
so $\|\ad(pqp)|_{Mp}\|\leq 2\gamma'$. As the vector $\zeta$ is cyclic for $Mp$ acting on the Hilbert space $p\Hs$, Proposition \ref{LocalDK}~(\ref{LocalDK:Part1}) gives an element $z\in (Mp)'=pM'p$ satisfying
$\|z-pqp\|\leq 2{(1+\sqrt{2})}\gamma'$. Replacing $z$ by $(z+z^*)/2$, we may assume that $z=z^*$.   Let $e_1\in pM'p$ be the spectral projection of $z$ for the
interval $[1-2{(1+\sqrt{2})}\gamma,1+2{(1+\sqrt{2})}\gamma']$.
Then $z(I_{p\Hs}-e_1)\leq (1-2{(1+\sqrt{2})}\gamma)I_{p\Hs}$. If $e_1\zeta=0$
then, since
$pqp\zeta=\zeta$ and $\langle pqp\zeta,\zeta\rangle =1$, we have a contradiction from
\begin{equation}
\langle pqp\zeta,\zeta\rangle = \langle (pqp-z)\zeta,\zeta\rangle+\langle
z(I_{p\Hs}-e_1)\zeta,\zeta\rangle
\leq 2{(1+\sqrt{2})}\gamma'+(1-2{(1+\sqrt{2})}\gamma)<1\label{RemarkEquation1001},
\end{equation}
as $\gamma'<\gamma$. Thus $e_1\zeta\ne 0$. From the functional calculus, $\|ze_1-e_1\|\leq
2{(1+\sqrt{2})}\gamma$, and so
\begin{equation}\label{eqdim1}
\|e_1-e_1qe_1\|\leq\|e_1-ze_1\|+\|ze_1-e_1qe_1\|=\|e_1-ze_1\|+\|e_1(z-pqp)e_1\|\leq
4{(1+\sqrt{2})}\gamma.
\end{equation}

There are now two cases to consider. If $M$ is not a   II$_1$ factor, then
rename $e_1$ as $e$, omitting the
next step. However, if $M$ is a   II$_1$ factor then $\dim_{Me_1}(e_1\Hs)\leq
1$ since $e_1\zeta$ is a
cyclic vector for $Me_1$ on $e_1\Hs$. Choose an integer $n$ so that
$n\dim_{Me_1}(e_1\Hs)\geq 1$ and choose a
projection $e\in e_1M'e_1$ so that
$\tau_{e_1M'e_1}(e)=(n\dim_{Me_1}(e_1\Hs))^{-1}$.
By Properties \ref{CC} (\ref{CC:Item1}),
$\dim_{Me}(e\Hs)=\tau_{e_1M'e_1}(e)\dim_{Me_1}(e_1\Hs)=1/n$,
and
\begin{equation}\label{eqdim2}
\|e-eqe\|=\|e(e_1-e_1qe_1)e\|\leq \|e_1-e_1qe_1\|\leq 4{(1+\sqrt{2})}\gamma
\end{equation}
from \eqref{eqdim1}. This shows that  \eqref{eqdim2} is satisfied in both cases.
Thus, from \eqref{eqdim2}, 
$
\|e(I_\Hs-q)e\|\leq 4{(1+\sqrt{2})}\gamma
$ so
$
\|(I_\Hs-q)e\|\leq (4{(1+\sqrt{2})}\gamma)^{1/2},
$ implying
$\|(I_\Hs-q)e(I_\Hs-q)\|\leq 4{(1+\sqrt{2})}\gamma.
$
Writing
$
e=qeq+qe(I_\Hs-q)+(I_\Hs-q)eq+(I_\Hs-q)e(I_\Hs-q),
$
we obtain the estimate
\begin{align}
\|e-qeq\|&\leq \|qe(I_\Hs-q)+(I_\Hs-q)eq\|+4{(1+\sqrt{2})}\gamma\notag\\
&=\max\{\|qe(I_\Hs-q)\|,\|(I_\Hs-q)eq\|\}+4{(1+\sqrt{2})}\gamma\notag\\
&\leq (4{(1+\sqrt{2})}\gamma)^{1/2}+4{(1+\sqrt{2})}\gamma.
\end{align}

If $x$ is in the unit ball of $N$, choose $y\in M$ with $\|x-y\|\leq \gamma$.
Noting that $x$ commutes with $q$ and $y$ commutes with both $e$ and $p$, the
algebraic
identity
\begin{align}
xq(qeq)&=qxeq=q(x-y)eq+qyeq
=q(x-y)eq+qeyq\notag\\
&=q(x-y)eq+qe(y-x)q+qeqxq
\end{align}
gives the inequalities $\|xq(qeq)-(qeq)xq\|\leq 2\gamma$ and
$\|\ad(qeq)|_{Nq}\|\leq 2\gamma$.
Since $\zeta$ is a cyclic vector for $Nq$ on $q\Hs$, Proposition
\ref{LocalDK}~(\ref{LocalDK:Part1}) gives $t\in (Nq)'=qN'q$ with $\|t-qeq\|\leq
2{(1+\sqrt{2})}\gamma$, so \eqref{eqCyclic1} is established by
\begin{equation}\label{eqCyclic3}
\|e-t\|\leq \|e-qeq\|+\|t-qeq\|\leq
2({(1+\sqrt{2})}\gamma)^{1/2}+6{(1+\sqrt{2})}\gamma.
\end{equation}
Define $\gamma_1=2({(1+\sqrt{2})}\gamma)^{1/2}+6{(1+\sqrt{2})}\gamma$, and now
suppose that the inequality
 $\gamma <1/87$ holds, which ensures that $\gamma_1<1/2$. Replacing $t$ by
$(t+t^*)/2$ if necessary, we may assume that $t$ is self-adjoint. The Hausdorff
distance between the spectra $\sp(e)$ and $\sp(t)$ is at most
$\|e-t\|\leq\gamma_1$ (see for example \cite[Proposition
2.1]{Dav:ActaSci}), and so  
 $\sp(t)$ is contained in
$[-\gamma_1,\gamma_1]\cup [1-\gamma_1,1+\gamma_1]$. If $f$ denotes the
spectral projection of $t$ for the second
of these intervals, then $\|f-t\|\leq \gamma_1$, giving the estimate
$\|e-f\|\leq 2\gamma_1<1$.
  Then Lemma \ref{ProjLem2} gives a
unitary $u\in (M'\cup N')''$ with $\|u-I_\Hs\|\leq \sqrt{2}\|e-f\|$ and $ueu^*=f$.
\end{proof}

\begin{remark}
Note that even in the II$_1$ case, the methods of Lemma \ref{Cyclic} do not enable us to take  $\tau_{M'}(e)$ close to $1$.
\end{remark}

We are now in a position to combine the previous results and show that, starting
with two close II$_1$ factors on a Hilbert space, it is possible to produce new
close representations of these factors on another Hilbert space so that both are
simultaneously in standard position. This enables us to transfer regular
amenable subalgebras from one factor to its close counterpart. The lemma below
does this in a form designed for immediate use in Section \ref{KKStable}. 

\begin{lemma}\label{Summation}
Let $M$ and $N$ be    $\text{\rm II}_1$ factors with separable preduals acting nondegenerately on
the Hilbert space
$\Hs$ and suppose that there is a positive constant $\gamma < 1.74\times 10^{-13}$
such
that $M\subset_\gamma N$ and
$N\subset_\gamma M$. Suppose further that there is an amenable von Neumann
subalgebra $P\subseteq M\cap N$ satisfying $P'\cap M\subseteq P$. Then there
exists a
Hilbert space $\Ks$ and faithful normal representations $\pi: M\to
\mathcal B(\Ks)$ and $\rho:N\rightarrow\mathcal B(\Ks)$
with the following properties:
\begin{enumerate}[\rm (i)]
\item\label{Sum:Item1} $\pi(M)$, $\pi(M)'$, $\rho(N)$ and $\rho(N)'$ are in standard position with common tracial vector $\xi$.
\item\label{Sum:Item2} $\pi(M) \subset_\beta\rho(N)$ and $\rho(N)\subset_\beta
\pi(M)$
for a constant $\beta< 50948\gamma^{1/2}<1/47.
$
\item\label{Sum:Item3} If $x\in M$ and $y\in N$ satisfy
$\|x\|, \|y\|\leq 1$, then
$\|\pi(x)-\rho(y)\|\leq\beta+\|x-y\|$. 
\item\label{Sum:Item4} $\pi|_P=\rho|_P$.
\item\label{Sum:Item5} $P'\cap N\subseteq P$.
\item\label{Sum:Item7} $P$ is regular in $M$ if and only if it is regular in
$N$.
\item\label{Sum:Item8} The basic construction algebras on $\Ks$ given by
$\langle
\pi(M),e_{\pi(P)}\rangle=(\pi(M)\cup\{e_{\pi(P)}\})''$ and $\langle
\rho(N),e_{\pi(P)}\rangle=(\rho(N)\cup\{e_{\pi(P)}\})''$ are equal, where
$e_{\pi(P)}$ is
the projection onto $\overline{\pi(P)\xi}$.
\item\label{Sum:Item11} If $P\subseteq M$ has  a bounded homogeneous orthonormal
basis of normalizers
$(u_n)_{n\geq 0}$ in $\N(P\subseteq M)$ and
$(v_n)_{n\geq 0}$ is a sequence in $\N(P\subseteq N)$ satisfying $\|u_n-v_n\|<
1-\beta$, then
$(v_n)_{n\geq 0}$ is a bounded homogeneous orthonormal basis of normalizers for
$P\subseteq
N$.
\end{enumerate}
\end{lemma}

\begin{proof} Use Popa's Theorem (Theorem \ref{PopaThm}) to choose a masa
$A\subseteq P$ such that $A$ is also a masa in $M$. Write $M_1=M$, $N_1=N$,
$P_1=P$ and $A_1=A$.  Since $\gamma <
1/87$,
we can apply Lemma \ref{Cyclic} to $M_1$ and $N_1$ to obtain nonzero projections
$e\in M_1'$ and $f\in
N_1'$ and a unitary $u_1\in
(M_1'\cup N_1')''$ so that $M_1e$ has a cyclic vector on $e\Hs$,
\begin{equation}
\|u_1-I_{\Hs}\| \le \sqrt 2\|e-f\| \leq \sqrt 2\left(12{(1+\sqrt{2})}\gamma +
4({(1+\sqrt{2})}\gamma)^{1/2}\right),
\end{equation}
and $u_1eu^*_1=f$. Moreover, we can additionally assume that
$\dim_{M_1e}(e\Hs)=1/n$ for an integer $n$. Since $e\in M_1'\cap
(u_1^*N_1u_1)'$ we can compress these algebras by $e$. Write $\Hs_2=e(\Hs)$,
$M_2=M_1e$, $N_2=(u_1^*N_1u_1)e$, $P_2=P_1e$ and $A_2=A_1e$ so that $M_2$ and
$N_2$ act on $\Hs_2$, $P_2\subseteq M_2\cap N_2$ (as $u_1$ commutes with $P_1)$
and the near inclusions $M_2\subset_{\gamma_2}N_2\subset_{\gamma_2}M_2$ hold,
where $
\gamma_2 = 2\sqrt 2\left(12{(1+\sqrt{2})}\gamma +
4({(1+\sqrt{2})}\gamma)^{1/2}\right)+\gamma <
(17.58)\gamma^{1/2}.
$
By construction $\dim_{M_2}(\Hs_2)=1/n$.
Now define $\Hs_3 = \Hs_2\otimes
{\mathbb C}^n$, $M_3=(M_2\otimes I_n)$,
$P_3=(P_2\otimes I_n)$, $A_3=(A_2\otimes I_n)$, $N_3=(N_2\otimes I_n)$ and
$\gamma_3=\gamma_2$. Then
$M_3\subset_{\gamma_3}N_3$ and
$N_3\subset_{\gamma_3}M_3$, $P_3\subseteq M_3\cap N_3$ and $\dim_{M_3}\Hs_3=1$
(by Properties \ref{CC} (\ref{CC:Item1}))
so that $M_3$ is in standard
position on $\Hs_3$.   

Consider the  masa $A_3\subseteq M_3$. As $\gamma_3<1/47$, Lemma \ref{Std.Lem1} (\ref{Std.Lem1:Part3})  gives $J_{M_3}M_3J_{M_3}\subset_{4{(1+\sqrt{2})}\gamma_3}N_3'$. As $4(1+\sqrt{2})\gamma_3<1/100$, another application of
Theorem \ref{Injective} (\ref{Injective:Part1}) provides a unitary $u_3\in (J_{M_3}A_3J_{M_3}\cup
N_3')''\subseteq P_3'$ so that
$\|I_{\Hs_3}-u_3\|<600{(1+\sqrt{2})}\gamma_3 $ and
$u_3(J_{M_3}A_3J_{M_3})u_3^*\subseteq N_3'$. Define $\Hs_4=\Hs_3$, $M_4=M_3$,
$P_4=P_3$, $A_4=A_3$, $N_4=u_3^*N_3u_3$, and
\begin{equation}
\gamma_4=(1200{(1+\sqrt{2})}+1)\gamma_3
<(1200{(1+\sqrt{2})}+1)(17.58)\gamma^{1/2}
<50948\gamma^{1/2}<1/47.
\end{equation}
Then $J_{M_4}A_4J_{M_4} \subseteq N_4'$, and $P_4\subseteq M_4\cap N_4$ since
$u_3$ commutes with $P_3$. The estimate (\ref{NearTriangleUnitary}) gives the
near
inclusion $M_4\subset_{\gamma_4} N_4\subset_{\gamma_4} M_4$ and then the bound
on $\gamma_4$
allows us to apply Lemma \ref{relcom2} to
conclude that $A_4$ is
also a masa in $N_4$. The hypotheses of Lemma \ref{Std.Lem2} are now met, from
which we see that 
 $M_4,M'_4$, $N_4$ and $N'_4$ have a common tracial vector.

At each stage of the proof, the various constructions have ensured that the
pairs $(M_k,M_{k+1})$ are canonically isomorphic via compressions,
amplifications or unitary conjugation, while the same is true for $(N_k,
N_{k+1})$, $1\le k\le 3$. Now let
$\Ks=\Hs_4$ and define  two isomorphisms $\pi\colon \ M\to M_4$ and $\rho\colon \
N\to N_4$ as the compositions of the
isomorphisms constructed above, whereupon $\pi(P) = P_4 \subseteq M_4\cap N_4 =
\pi(M)\cap \rho(N)$. We take
$\beta$ to be  $\gamma_4$. This establishes parts (\ref{Sum:Item1}) and
(\ref{Sum:Item2}). Since all the unitaries $u_k$ used to construct the
isomorphisms $\pi$ and $\rho$ commute with the corresponding $P_k$, it follows
directly that $\pi|_P=\rho|_P$,  giving (\ref{Sum:Item4}). The estimate of part
(\ref{Sum:Item3}) also follows from repeated applications of the triangle inequality given the explicit form of the $\pi$ and $\rho$ as
a composition of unitary conjugations, compressions and amplifications. 

Item (\ref{Sum:Item5}) is Lemma \ref{relcom2}, while the
remaining conditions follow from Lemma \ref{generate}.  Indeed (\ref{Sum:Item8})
is part (\ref{generate:Part2}) of Lemma \ref{generate}, while part (\ref{generate:Part4}) of Lemma \ref{generate}
shows that $\pi(P)$ is regular in $\pi(M)$ if and only if $\rho(P)$ is regular
in $\rho(N)$. Since $\pi$ and $\rho$ are both faithful, condition
(\ref{Sum:Item7}) follows. For (\ref{Sum:Item11}), given $(u_n)$ and $(v_n)$ as
in this condition, note that the hypothesis $\|u_n-v_n\|<1-\beta$ gives
$\|\pi(u_n)-\rho(v_n)\|<1$ by part (\ref{Sum:Item3}). Thus, as $(\pi(u_n))_n$
is a bounded homogeneous orthonormal basis of normalizers for $\pi(P)\subseteq
\pi(M)$, Lemma \ref{generate} (\ref{generate:Part3}) shows that $(\rho(v_n))_n$ is a bounded
homogeneous orthonormal basis of normalizers for $\rho(P)\subseteq \rho(N)$ so that $(v_n)_n$ has the same property for $P\subseteq N$.
\end{proof}

The procedure above enables us to transfer close   II$_1$ factors on a
Hilbert space to
another space so that they both act in standard position.  Thus, for the weakest version of the 
Kadison-Kastler stability problem for   II$_1$ factors, we can assume that
all factors act in standard position: we record this below. More
care is required for the stronger spatial versions of the problem. Any isomorphism between close II$_1$ factors will be spatially implemented on the new Hilbert space but,
without a positive answer to Question \ref{QNew}, we do not know whether this isomorphism is spatially
implemented on the original Hilbert space.

\begin{theorem}\label{ReductionResult}
Let $M$ and $N$ be   ${\mathrm{II}}_1$ factors with separable preduals nondegenerately represented
on a Hilbert space $\Hs$ with $M\subset_\gamma N$ and $N\subset_\gamma M$ for
some $\gamma<5.7\times 10^{-16}$. Then there exists a Hilbert space $\Ks$ and
faithful normal representations $\pi:M\rightarrow \mathcal B(\Ks)$ and
$\rho:N\rightarrow \mathcal B(\Ks)$ such that $\pi(M)$, $\pi(M)'$, $\rho(N)$ and
$\rho(N)'$ have a common tracial vector on $\Ks$ and
$\pi(M)\subset_\beta\rho(N)\subset_\beta\pi(M)$ for 
$
\beta=50948\times(301)^{1/2}\gamma^{1/2}+300\gamma< 8.84\times
10^{5}\gamma^{1/2}.
$
 For $x\in M$ and $y\in N$  with $\|x\|,\|y\|\leq 1$, we have
$\|\pi(x)-\rho(y)\|\leq\beta+300\gamma+\|x-y\|$.
\end{theorem}
\begin{proof}
Choose a masa $A$ in $M$. By Theorem \ref{Injective} (i) there is a unitary
$u\in
(A\cup N)''$ with $\|u-I_{\Hs}\|\leq 150\gamma$ and $uAu^*\subseteq N$. Consider
$N_1=u^*Nu$, which has $M\subset_{\gamma_1}N_1\subset_{\gamma_1}M$, where
$\gamma_1=301\gamma$. As $\gamma_1<1.74\times 10^{-13}$, we can take $M$ and
$N_1$
in Lemma \ref{Summation} to obtain representations $\pi$ and $\rho_1$ satisfying
the properties of that lemma (with the $\beta$ of that lemma being given by $50948\times(301)^{1/2}\gamma^{1/2}$).  Define $\rho(y)=\rho_1(u^*yu)$. It is routine to
verify  that $\pi$ and $\rho$ satisfy the required estimates.
\end{proof}

Since completely close II$_1$ factors, have (completely) close commutants (Proposition \ref{cb-com}), the process of changing representations is much easier in this context. We start by noting that completely close II$_1$ factors always have the same coupling constant.

\begin{proposition}\label{cordim}
Let $M$ and $N$ be   ${\mathrm{II}}_1$ factors acting nondegenerately on a separable
Hilbert space $\Hs$. If $d_{cb}(M,N)<\gamma<(301\times 136209)^{-1}$,
then $\dim_M(\Hs)=\dim_N(\Hs)$. 
\end{proposition}

\begin{proof}
This is obvious if both dimensions are infinite, so suppose
 that
$\dim_M(\Hs)<\infty$.  We have $M\subset_{cb,\gamma}N$ and
$N\subset_{cb,\gamma}M$ and can further assume that $\dim_M(\Hs)\leq 1$, as otherwise we can use Proposition \ref{cb-com} to replace $M$ and $N$ with $M'$ and $N'$ respectively.

Choose a masa $A\subseteq M$. Since $\gamma<1/100$, Theorem \ref{Injective} (\ref{Injective:Part1})
gives a unitary $u\in (A\cup N)''$ so that $\|u-I_\Hs\|\leq 150\gamma$, $d(A,
uAu^*)\leq 100\gamma$, and $uAu^*\subseteq N$. Define $N_1=u^*Nu$ and note that
$A\subseteq M\cap N_1$ while $d(M,N_1)\leq d(M,N)+2\|u-I_\Hs\|\leq 301\gamma$.
Since $301\gamma<1$, Lemma \ref{relcom2} allows us to
conclude that $A$ is also a masa in $N_1$. The hypotheses of Proposition
\ref{dim} are now met since $301\gamma<1/136209$ and so
$\dim_{N_1}(\Hs)=\dim_M(\Hs)$. The result follows since
$\dim_{N_1}(\Hs)=\dim_N(\Hs)$ by the unitary conjugacy of $N$ and $N_1$.
\end{proof}

\begin{proposition}\label{Gamma-Spatial}
Let $M$ be a  weakly Kadison-Kastler stable II$_1$ factor with property $\Gamma$ and separable predual. Then $M$ is Kadison-Kastler stable.
\end{proposition}
\begin{proof}
Suppose that $M\subseteq\mathcal B(\Hs)$ is a nondegenerate normal
representation and that $N\subseteq\mathcal B(\Hs)$ has $d(M,N)$ small enough so
that $M\cong N$. Then $N$ also has property $\Gamma$, and so Proposition \ref{DK}
(\ref{DK:Part2}) and standard properties of $d_{\cb}$ give $d_{\cb}(M,N)\leq 10d(M,N)$.  If
additionally $d(M,N)<1/(10\times 301\times 136209)$,  then Proposition
\ref{cordim} gives $\dim_M(\Hs)=\dim_N(\Hs)$ and so there is a spatial isomorphism between $M$
and $N$  \cite{PdelaB} (see also \cite[\S
6.4, Prop. 10]{Dix:Book}).
\end{proof}

\begin{proposition}\label{reps}
Let $M$ and $N$ be von Neumann algebras acting nondegenerately
on a Hilbert space 
$\Hs$ with $d_{cb}(M,N)<\gamma$ for $\gamma<1$. Given any unital normal
$^*$-representation $\pi$ of $M$ on another Hilbert space $\Ks$, there exists a
unital normal $^*$-representation $\rho$ of $N$ on $\Ks$ such that 
$d_{cb}(\pi(M),\rho(N))\leq 3\gamma$ and $\rho|_{M\cap N}=\pi|_{M\cap N}$. Further, if $x\in M$ and $y\in N$ are contractions, then $\|\pi(x)-\rho(y)\|\leq 2\gamma+\|x-y\|$. 
\end{proposition}
\begin{proof}
The general theory of normal *-representations \cite[I \S{4} Theorem
3]{Dix:Book} allows us to choose a set $S$ and a projection $p\in
M'\vnotimes{\mathcal B}(\ell^2(S))$ so that $\pi$ is unitarily
equivalent to the *-representation
$\pi_1: x\mapsto (x\otimes I_{\ell^2(S)})p$ of $M$ on $\Ks_1=p(\Ks)$. Identifying $M$ and $N$ with their
amplifications $M\otimes
I_{\ell^2(S)}$ and $N\otimes I_{\ell^2(S)}$ respectively and, noting that
$M'\,\overline{\otimes}\,{\mathcal
B}(\ell^2(S))\subset_{\gamma}N'\,\overline{\otimes}\,{\mathcal B}(\ell^2(S))$ by
Proposition \ref{cb-com}, it
follows from Lemma \ref{ProjLem} (\ref{ProjLem:Part2}) that there is a projection $q\in
N'\,\overline{\otimes}\,{\mathcal B}(\ell^2(S))$ with $\|p-q\|<2^{-1/2}\gamma$. Since $p,q\in (M\cap N)'\vnotimes\mathcal B(\ell^2(S))$,  Lemma \ref{ProjLem2} gives a unitary $u\in (M\cap N)'\vnotimes\mathcal B(\ell^2(S))$ such that
$\|u-I_{\Hs\otimes \ell^2(S)}\|<\gamma$ and $upu^*=q$. Define $\rho_1:N\to
{\mathcal B}(\Ks_1)$ by
$\rho_1(x)=u^*(x\otimes I_{\Hs\otimes\ell^2(S)})u$ for $x\in N$. By construction $\rho_1|_{M\cap N}=\pi_1|_{M\cap N}$. Further
$
d_{cb}(\pi_1(M),\rho_1(N))\leq d_{cb}(M,N)+2\|u-I_{\Hs\otimes
\ell^2(S)}\|<3\gamma.
$
Pick a unitary $V:\Ks\to \Ks_1$ so that $\pi=V^*\pi_1V$, and let
$\rho=V^*\rho_1V$ so that $d_{\cb}(\pi(M),\rho(N))<3\gamma$. For contractions $x\in M$ and $y\in N$, we have $\|\pi(x)-\rho(y)\|\leq 2\|u-I_{\Hs\otimes\ell^2(S)}\|+\|x-y\|$. 
\end{proof}

In the presence of complete closeness, we also obtain a more direct proof of the key reduction result Lemma \ref{Summation} with improved constants.   Note that the constant $\beta$ below is now $O(\gamma)$ as $\gamma \to 0$ whereas it was $O(\gamma^{1/2})$ in the original (c.f. Remark (1) of \cite{CCSSWW:InPrep}).

\begin{theorem}\label{Summation2}
Let $M$ and $N$ be   ${\mathrm{II}}_1$ factors with separable preduals acting
nondegenerately on a Hilbert
space $\Hs$ and suppose that there is a constant $\gamma<1/(903\times 47)$ such
that $M\subset_{\cb,\gamma}N$ and $N\subset_{\cb,\gamma}M$. Suppose that
$P\subseteq M\cap N$ is an amenable von Neumann algebra satisfying $P'\cap
M\subseteq P$.. Then there exist a separable Hilbert space $\Ks$ and faithful
normal
*-representations $\pi: M\to {\mathcal B}(\Ks)$ and $\rho:N\to {\mathcal B}(\Ks)$ such that:
\begin{enumerate}[(i)]
\item Property (\ref{Sum:Item1}) from Lemma \ref{Summation} is satisfied.
\item\label{Summation2:C2} $\pi(M)\subset_{\cb,\beta}\rho(N)$ and
$\rho(N)\subset_{\cb,\beta}\pi(M)$ where $\beta=903\gamma$.
\item\label{Summation2:C3} For contractions $x\in M$ and $y\in N$, $\|\pi(x)-\rho(y)\|\leq903\gamma+\|x-y\|$.
\item $\pi|_P=\rho|_P$;
\item Properties (\ref{Sum:Item5}-\ref{Sum:Item11}) from Lemma \ref{Summation} are satisfied (with the value of $\beta$ above).
\end{enumerate}
\end{theorem}

\begin{proof}
By Proposition \ref{reps}, take faithful representations
$\pi:M\rightarrow\mathcal B(\Ks)$ and $\rho_1:N\rightarrow\mathcal B(\Ks)$
which agree on $P$ such that $M_1=\pi(M)$ is in standard position on $\Ks$,
$M_1\subset_{\cb,3\gamma}\rho_1(N)$, $\rho_1(N)\subset_{\cb,3\gamma}M_1$ and
$\|\pi_1(x)-\rho_1(y)\|\leq 3\gamma$ whenever $x\in M$ and $y\in N$ are
contractions with $\|x-y\|<\gamma$. Write $P_1=\pi_1(P)=\rho_1(P)$ and
$N_1=\rho_1(N)$ and fix, by Popa's theorem (\cite{P:Invent}), a masa $A_1\subset
M_1$ with $A_1\subseteq P_1$. Since $M_1'\subset_{3\gamma}N_1'$ by Proposition
\ref{cb-com}, apply Theorem \ref{Injective} (i) to obtain
a unitary $u\in (J_{M_1}A_1J_{M_1}\cup N_1')''$ so that
$\|u-I_{\Ks}\|<450\gamma$ and $uJ_{M_1}A_1J_{M_1}u^*\subseteq N_1'$. Define $\rho:N\rightarrow\mathcal B(\Ks)$ by $\rho(y)=u^*\rho_1(y)u$ so that conditions (\ref{Summation2:C2}) and (\ref{Summation2:C3}) hold.  Property (\ref{Sum:Item1}) from
Lemma \ref{Summation} follows from Lemma \ref{Std.Lem2}, as the estimate on $\gamma$ ensures that $\beta<1/47$.
Property (\ref{Sum:Item5}) from Lemma \ref{Summation} is now obtained from  Lemma \ref{relcom2}  while the rest follow from Lemma \ref{generate}  as in the proof of Lemma \ref{Summation}.
\end{proof}

\section{McDuff factors}\label{properties}

In this short section we show that factors sufficiently close to McDuff factors are also McDuff, and that, after making a small unitary perturbation, it is possible to simultaneously factorize both algebras.  We use this in the proof of Theorem \ref{TA}, as the factors there are McDuff. Recall that if $S$ is a matrix subalgebra of a II$_1$ factor $N$, then $N\cong (S'\cap N)\vnotimes S$ and $N$ is algebraically generated by $S'\cap N$ and $S$.  Further, if $N$ is generated as a von Neumann algebra by two commuting subfactors $Q_1$ and $Q_2$, then $N\cong Q_1\vnotimes Q_2$, \cite{N:TMJ}.

\begin{lemma}\label{McDuffLem2}
Let $M$ be a McDuff factor with a separable predual acting nondegenerately on a
Hilbert space $\Hs$ and write $M=M_0\vnotimes R$ for some II$_1$ factor $M_0$ on
$\Hs$, where $R$ is the hyperfinite II$_1$ factor.  Suppose that $N$ is another II$_1$ factor on $\Hs$ with
$d(M,N)<\gamma<1/602$. Given an amenable subalgebra $P_0$ of
$M_0$ with $P_0'\cap M_0\subseteq P_0$, there exists a unitary $u\in (M\cup
N)''$
with $\|u-I_\Hs\|\leq150\gamma$ such that:
\begin{enumerate}[(i)]
\item\label{McDuffLem:Part1} writing $R_1=uRu^*$, we have $R_1\subseteq N$;
\item\label{McDuffLem:Part2} $N$ is generated by $R_1$ and $R_1'\cap N$, so that $N\cong (R_1'\cap N)\vnotimes R_1$;
\item\label{McDuffLem:Part3} $R_1'\cap N\subseteq_{\cb,(200\sqrt{2}+5)\gamma}M_0$ and $M_0\subseteq_{\cb,(200\sqrt{2}+5)\gamma}R_1'\cap N$;
\item\label{McDuffLem:Part4} $uP_0u^*\subseteq R_1'\cap N$.
\end{enumerate}
In particular $N$ is McDuff.
\end{lemma}

\begin{proof}
First suppose that $P:=(P_0\cup R)''\subseteq N$ and that $\gamma<1/2$. Let $S$ be a matrix subfactor of $R$ and decompose $R$ as $Q \otimes S$ where $Q$ is a subfactor of $R$. Then $M_0\otimes S=Q'\cap M$ and $d(Q'\cap M,Q'\cap N)\leq \gamma$ by Lemma \ref{RelCom} (i). Fix $y\in M_0\otimes S$ with $\|y\|\leq 1$ and choose $z\in Q'\cap N$
so that $\|z\|\leq 1$ and $\|y-z\|\leq \gamma$. Since $S$ and $S'\cap N$ generate $N$ as an algebra, choose $z_i\in S'\cap N$ and $s_i\in S$ so that $z=\sum_{i=1}^k z_is_i$. Since $y$
and the $s_i$'s commute with $Q$, $\|y-\sum_{i=1}^kuz_iu^*s_i\|\leq \gamma$ for $u\in \mathcal{U}(Q)$, and hence $\|y-\sum_{i=1}^kE^N_{Q'\cap N}(z_i)s_i\|\leq\gamma$, where $E^N_{Q'\cap N}:N\rightarrow Q'\cap N$ is the trace-preserving conditional expectation. As $z_i\in S'\cap N$, we have $E^N_{Q'\cap N}(z_i)\in (Q\cup S)'\cap N=R'\cap N$.  Thus $Q'\cap M\subseteq_\gamma ((R'\cap N)\cup S)''\subseteq ((R'\cap N)\cup R)''$.

Now consider $x\in N$, $\|x\|\leq 1$ and choose $y\in M$, $\|y\|\leq 1$, so that $\|y-x\|\leq \gamma$. Letting $(S_n)_{n=1}^\infty$ be an increasing sequence of matrix
subalgebras which is weak$^*$-dense in $R$, there is a sequence of contractions $(y_n)_{n=1}^\infty$ with $y_n\in M_0\otimes S_n$, converging weak$^*$ to $y$. Choose elements
$\tilde{z}_n\in ((R'\cap N)\cup R)''$ so that $\|y_n-\tilde{z}_n\|\leq \gamma$, and let $\tilde{z}$ be a $w^*$-accumulation point. Then $\|y-\tilde{z}\|\leq \gamma$ so $\|x-\tilde{z}\|\leq 2\gamma$. Thus $N\subseteq_{2\gamma}((R'\cap N)\cup R)''\subseteq N$, so $N=((R'\cap N)\cup R)''$ since $2\gamma
< 1$ (see \cite[Proposition 2.4]{CSSWW:Acta}). Since $N$ is a factor generated by $R'\cap N$ and $R$, it follows that $R'\cap N$ is also a factor, whence $N\cong (R'\cap N)\vnotimes R$, so $N$ is McDuff.

For the general case, Theorem 3.3 (i) gives a unitary $u\in (P\cup N)''$ with $\|u-I_\Hs\|\leq 150 \gamma$ and $\|uxu^*-x\|\leq 100\gamma\|x\|$ for $x\in P$ so that
$uPu^*\subseteq N$. Let $N_1=u^*Nu$. Then $P\subseteq M\cap N_1$ and $d(M,N_1)\leq 301\gamma$. Since $301\gamma<1/2$, the first part applies to $M$ and $N_1$, showing
that $N_1$, and hence $N$, is McDuff. Now Proposition \ref{DK} (ii) applies to give $N\subset_{cb,5\gamma}M\subset_{cb,5\gamma} N$. As $R_1:=uRu^*\subseteq_{100\gamma}R$, Lemma \ref{RelCom} (\ref{RelCom:Part2}) gives $M_0=R'\cap M\subseteq_{\cb,200\sqrt{2}\gamma+5\gamma}(uRu^*)'\cap N$.  Similarly, $R_1'\cap N=(uRu^*)'\cap N\subseteq_{\cb,200\sqrt{2}\gamma+5\gamma}R'\cap M$, and all parts of the lemma have
been established.
\end{proof}

\begin{corollary}\label{McDuffStable}
Suppose that $M$ is a weakly Kadison-Kastler stable II$_1$ factor. Then $M\vnotimes R$ is Kadison-Kastler stable, where $R$ is the hyperfinite II$_1$ factor.
\end{corollary}
\begin{proof}
If $d(M\vnotimes R,N)<\gamma<1/602$, then by Lemma \ref{McDuffLem2} we can factorize $N=N_0\vnotimes R_1$ for some factor $N_0$ with $M\subseteq_{\cb,(200\sqrt{2}+5)\gamma}N_0$ and $N_0\subseteq_{\cb,(200\sqrt{2}+5)\gamma}M$. Thus if $\gamma$ is small enough, then $M\cong N_0$, and hence $M\vnotimes R\cong N$, that is $M\vnotimes R$ is weakly Kadison-Kastler stable.   Since $M\vnotimes R$ has property $\Gamma$ it is Kadison-Kastler stable by Proposition \ref{Gamma-Spatial}.
\end{proof}

\section{Kadison-Kastler stable factors}\label{KKStable}

We now present the main results of the paper by exhibiting classes
of actions giving rise to Kadison-Kastler stable crossed product factors. The first step is to combine our earlier work to transfer a twisted crossed product structure from a II$_1$ factor to nearby factors.  Recall that every
$2$-cocycle is cohomologous to a normalized $2$-cocycle so 
there is no loss of
generality in only considering normalized cocycles below. 
\begin{theorem}\label{cocyclelem}
Let $\alpha$ be a trace preserving, centrally
ergodic, and properly outer action  of a countable
discrete group $\Gamma$ on a finite amenable von Neumann algebra $P$ with  separable predual.
Let $\omega\in Z^2(\Gamma,\U(\Z(P)))$ be a normalized $2$-cocycle, let
$M=P\rtimes_{\alpha,\omega}\Gamma$ and take a  unital normal representation $M\subseteq\mathcal B(\Hs)$.
 Let  $N\subseteq \mathcal B(\Hs)$ be a von
Neumann algebra
 with $d(M,N)<\gamma<5.77\times 10^{-16}$.  Then $N\cong
P\rtimes_{\alpha,\omega'}\Gamma$ for a normalized
 $\omega'\in Z^2(\Gamma,\U(\Z(P)))$ with 
\begin{equation}\label{cocycleeqn}
\sup_{g,h\in \Gamma}\|\omega(g,h)-\omega'(g,h)\|<14889\gamma< 8.6\times
10^{-12}.
\end{equation}
\end{theorem}

\begin{proof}
By Proposition \ref{TCPFactorProp}, $M$ is a II$_1$ factor and $P'\cap
M\subseteq P$.  The bound on $\gamma$ ensures that Kadison and Kastler's
stability of type classification from \cite{KK:AJM} applies and so $N$ is also a
II$_1$ factor.   Since $\gamma<1/100$, Theorem \ref{Injective} (\ref{Injective:Part1}) provides a
unitary $u\in (P\cup N)''$ with $\|I_\Hs-u\|\leq 150\gamma$ such that
$P\subseteq N_1:=u^*Nu$.  Moreover, we have
$M\subset_{\gamma_1}N_1\subset_{\gamma_1}M$, where $\gamma_1=301\gamma$.   Write
$(u_g)_{g\in \Gamma}$ for the canonical
bounded homogeneous orthonormal basis of normalizers for $P\subseteq M$ implementing the action
$\alpha$
which satisfy $u_gu_h=\omega(g,h)u_{gh}$ for $g,h\in \Gamma$.  For $g\in \Gamma$, we can
apply Lemma \ref{NormaliserLemma0} (\ref{NormaliserLemma0:Item3}) to $M$ and $N_1$  to obtain a
normalizer
$w_g\in\N(P\subseteq N_1)$ with 
$
\|w_g-u_g\|\leq (4+2\sqrt{2})\gamma_1<1.
$
Thus, by  Proposition \ref{InjectiveUnitary}, $w_g=u_gp_gp_g'$ for some
unitaries $p_g\in P$ and $p_g'\in P'$ with $\|p_g-I_{\Hs}\|\leq
2^{1/2}\|w_g-u_g\|\leq2^{1/2}(4+2\sqrt{2})\gamma_1$. Write $v_g=w_gp_g^*=u_gp_g'\in N_1$ so that
$v_gxv_g^*=u_gxu_g^*=\alpha_g(x)$ for all $x\in P$ and we have
$
\|v_g-u_g\|\leq (1+\sqrt{2})(4+2\sqrt{2})\gamma_1=(8+6\sqrt{2})\gamma_1<4963\gamma.
$
Since $u_e=I_\Hs$, we may assume that $v_e=w_e=I_\Hs$. For use in Lemma \ref{BoundedLem}, note that if $N$ happens
to already contain $P$ then we can take $N=N_1$ and $u=I_\Hs$, and the
estimate above is replaced by 
\begin{equation}\label{cocycle:eqn2}
\|v_g-u_g\|\leq (8+6\sqrt{2})\gamma<16.5\gamma.
\end{equation}

The bound on $\gamma$ in the statement of the theorem is chosen so that
$\gamma_1<1.74\times 10^{-13}$ and so Lemma \ref{Summation} applies as $P'\cap M\subseteq P$.  In
particular
$P'\cap N_1\subseteq P$, while  Lemma \ref{Summation} (\ref{Sum:Item11})
and the first estimate  show that $(w_g)_{g\in \Gamma}$ is a bounded
homogeneous orthonormal basis of normalizers for $P\subseteq N_1$. Then
 $(v_g)_{g\in \Gamma}$ also has this property and so $N_1$ is generated by $P$
and the normalizers $(v_g)_{g\in \Gamma}$. As we have  $E_P^{N_1}(v_g)=0$ for
$g\in
\Gamma\setminus \{e\}$ and $v_gxv_g^*=\alpha_g(x)$ for $x\in P$, Proposition
\ref{TCPProp} shows that $N\cong N_1\cong P\rtimes_{\alpha,\omega'}\Gamma$,  where $
\omega'(g,h)=v_gv_hv_{gh}^*\in \U(\Z(P))$ for $g,h\in \Gamma.
$
Since we chose $v_e=I_{\Hs}$, this $2$-cocycle is normalized and using the
estimate $\|v_g-u_g\|<4963\gamma$ three times gives
\begin{align}
\|\omega(g,h)-\omega'(g,h)\|&=\|u_gu_hu_{gh}^*-v_gv_hv_{gh}^*\|
\leq\|u_g-v_g\|+\|u_h-v_h\|+\|u_{gh}^*-v_{gh}^*\|\nonumber\\
&<3\times 4963\gamma=14889\gamma
<8.6\times 10^{-12},\quad g,h\in \Gamma\label{cocycle:eqn3},
\end{align}
as required.
\end{proof}

We obtain weakly Kadison-Kastler stable factors whenever we can guarantee that
the uniformly close cocycles $\omega$ and $\omega'$ in Theorem \ref{cocyclelem}
are cohomologous. The next corollary proves Part
(\ref{TB1}) of Theorem \ref{TB}, noting that if either  of the comparison maps $H^2_b(\Gamma,\Z(P)_{sa})\rightarrow H^2(\Gamma,\Z(P)_{sa})$ and $H^2_b(\Gamma,L^2(\Z(P)_{sa}))\rightarrow H^2(\Gamma,L^2(\Z(P)_{sa}))$ vanishes, then the same is true for the map in  (\ref{ComMap}).

\begin{corollary}\label{MainCor}
Let $\alpha$ be a trace preserving, centrally
ergodic, and properly outer action  of a countable
discrete group $\Gamma$ on a finite amenable von Neumann algebra $P$ with
 separable predual. Suppose that the comparison map
\begin{align}
H^2_b(\Gamma,\Z(P)_{sa})\rightarrow H^2(\Gamma,L^2(\Z(P)_{sa}))\label{ComMap}
\end{align}
vanishes. Fix a normalized $2$-cocycle $\omega\in Z^2(\Gamma,\U(\Z(P)))$ and let $M$ be the
  $\text{\rm II}_1$ factor $P\rtimes_{\alpha,\omega}\Gamma$,
faithfully, normally, and nondegenerately represented on a Hilbert space $\Hs$.
Then  $M\cong N$ whenever $N$ is a von Neumann algebra on $\Hs$
with $d(M,N)< 5.77 \times 10^{-16}$, and so $M$ is weakly Kadison-Kastler stable.
\end{corollary}
\begin{proof}
Let $M:=P\rtimes_{\alpha,\omega}\Gamma$ for a
normalized 2-cocycle $\omega$ and suppose that $M$ is faithfully normally and
nondegenerately represented on $\Hs$. Given another von Neumann algebra $N$ on
$\Hs$ with $d(M,N)<5.77\times 10^{-16}$, we have $N\cong
P\rtimes_{\alpha,\omega'}\Gamma$ for some normalized
$\omega'\in Z^2(\Gamma,\U(\Z(P)))$ satisfying (\ref{cocycleeqn}) by Theorem
\ref{cocyclelem}. Define a $2$-cocycle $\nu\in Z^2(\Gamma,\U(\Z(P)))$ by $\nu(g,h)=\omega(g,h)\omega'(g,h)^*$ so that 
\begin{equation}
\sup_{g,h\in\Gamma}\|\nu(g,h)-I_P\|<8.6\times 10^{-12}<\sqrt{2}.
\end{equation}
By Lemma \ref{LogLem} (\ref{LogLem:Part1}), $\psi:=-i\log\nu$ is a bounded $2$-cocycle
in $Z^2_b(\Gamma,\Z(P)_{sa})$. As the map of (\ref{ComMap}) vanishes, $\psi=\partial\phi$ for some $\phi\in C^1(\Gamma,L^2(\Z(P)_{sa}))$. Then $\nu=\partial e^{i\phi}$ by Lemma \ref{LogLem} (\ref{LogLem:Part3}), hence
$\omega$ and $\omega'$ are cohomologous. Thus $M\cong N$ by
Proposition \ref{TCPProp} (\ref{TCPProp:Part2}).  
\end{proof}

In particular the previous result applies when $\Gamma$ is a free group, as $H^2(\Gamma,\Z(P)_\mathrm{sa})=0$.  
\begin{corollary}\label{KKfree}
Let $\mathbb F_r$ be a free group of rank $r=2,3,\ldots,\infty$ and let
$\alpha:\mathbb F_r\curvearrowright P$ be a trace preserving, centrally ergodic,
properly outer action on a finite amenable von Neumann algebra $P$ with
separable predual. Then $P\rtimes_{\alpha}\mathbb F_r$ is weakly Kadison-Kastler
stable.
\end{corollary}

To obtain Kadison-Kastler stable factors from Corollary \ref{MainCor}, we need to impose additional conditions using the similarity property  to ensure that the
isomorphism is spatially implemented.  The corollary
below is immediate from Corollary \ref{MainCor} and  Proposition \ref{Gamma-Spatial} and proves Part (\ref{TB2}) of Theorem \ref{TB}.

\begin{corollary}\label{spatialKK}
Let $\alpha:\Gamma\curvearrowright P$ be a properly outer, centrally ergodic,
trace preserving action of a countable discrete group on a finite amenable von
Neumann algebra $P$ with separable predual. Further, suppose that the
comparison map of (\ref{ComMap}) vanishes. If the crossed product factor
$P\rtimes_\alpha \Gamma$ has property $\Gamma$, then it is Kadison-Kastler stable.
\end{corollary}

Taking $\Gamma$ to be a free group, we obtain examples of the previous corollary when the action additionally is not strongly ergodic (see \cite{CW:IJM,Sch:ETDS}), as asymptotically invariant subsets of $X$ (\cite[Lemma 1]{HJ:OAMP}) give rise to central sequences for $L^\infty(X)\rtimes\Gamma$.  As noted
in \cite[Section 5]{OP:Ann} one can construct non-strongly ergodic actions of free groups which are additionally profinite.
\begin{corollary}
Let $\alpha:\mathbb F_r\curvearrowright(X,\mu)$ be a free, ergodic, probability measure preserving action which is not strongly ergodic.  Then $L^\infty(X,\mu)\rtimes_\alpha\Gamma$ is  Kadison-Kastler stable.
\end{corollary}
When $\Gamma$ is not inner amenable, failure of strong ergodicity is the only way that the crossed product factor can have property $\Gamma$ (\cite[Lemma 1]{HJ:OAMP}). As property (T) is an obstruction to the existence of ergodic actions which fail to be strongly ergodic, we cannot obtain examples using $SL_n(\mathbb Z)$ for $n\geq 3$ in this way.

 Corollaries \ref{MainCor} and  \ref{McDuffStable} imply  that the tensor product of each weakly Kadison-Kastler stable factor above with the hyperfinite II$_1$ factor is automatically Kadison-Kastler stable.
\begin{corollary}
Let $\alpha:\Gamma\curvearrowright P$ be a trace preserving, centrally
ergodic, and properly outer action  of a countable
discrete group $\Gamma$ on a finite amenable von Neumann algebra $P$ with
separable predual. Suppose that  the comparison map of (\ref{ComMap}) vanishes (as happens when $\Gamma$ is a free group), and write $M=P\rtimes_\alpha\Gamma$.  Then the II$_1$ factor $M\vnotimes R$ is Kadison-Kastler stable, where $R$ is the hyperfinite II$_1$ factor.
\end{corollary}

We now turn to the situation where the bounded cohomology groups
$H^2_b(\Gamma,\Z(P)_{sa})$ vanish, and  we first examine the case  when the crossed product factor lies in standard position. We thank one of the referees for a significant simplification of the proof of the next lemma.

\begin{lemma}\label{BoundedLem}
Let $\alpha:\Gamma\curvearrowright P$ be a trace preserving, centrally ergodic
and properly outer action of a countable discrete group $\Gamma$ on a finite
amenable von Neumann algebra $P$ with separable predual. Suppose that
$H^2_b(\Gamma,\Z(P)_{sa})=0$.   Given a normalized
$2$-cocycle $\omega\in Z^2(\Gamma,\U(\Z(P)))$, write
$M=P\rtimes_{\alpha,\omega}\Gamma$ and suppose that $M\subseteq\mathcal B(\Ks)$
is represented in standard position with tracial vector $\xi$ used to
define the modular conjugation operator $J_M$ and the orthogonal projection
$e_P$ onto $\overline{P\xi}$. Let $N\subseteq\mathcal B(\Ks)$ be another von
Neumann algebra with $M\subseteq_\beta N$ and $N\subseteq_\beta M$ for $\beta<1/47$ and such that
$P\subseteq N$ and $J_MPJ_M\subseteq N'$. Then there exists a unitary $U\in
\mathcal B(\Ks)$ such that $UMU^*=N$,  and 
\begin{equation}
\|U-I_\Ks\|\leq(170+114\sqrt{2})\beta<333\beta.
\end{equation}
\end{lemma}

\begin{proof}
Write $(u_g)_{g\in\Gamma}$ for the canonical unitaries in
$M=P\rtimes_{\alpha,\omega}\Gamma$ satisfying $u_gu_h=\omega(g,h)u_{gh}$ for
$g,h\in\Gamma$.  Just as in the proof of Theorem \ref{cocyclelem} (see
equation (\ref{cocycle:eqn2})), we can find unitaries $(v_g)_{g\in\Gamma}$ in $N$ satisfying
$
\|v_g-u_g\|\leq (8+6\sqrt{2})\beta
$
such that $v_e=u_e=I_{\Ks}$, $v_gxv_g^*=u_gxu_g^*=\alpha_g(x)$ for $x\in P$ and $(v_g)_{g\in\Gamma}$ forms a bounded homogeneous orthonormal basis of
normalizers for $P\subseteq N$. By Proposition \ref{TCPProp},  $N\cong
P\rtimes_{\alpha,\omega'}\Gamma$ where $\omega'$ is the normalized $2$-cocycle
given by $\omega'(g,h)=v_gv_hv_{gh}^*$.  Then
$\nu(g,h)=\omega(g,h)\omega'(g,h)^*$ has
$\sup_{g,h\in\Gamma}\|\nu(g,h)-I_P\|\leq3(8+6\sqrt{2})\beta$ following the argument of
(\ref{cocycle:eqn3}). Thus, defining $\psi=-i\log\nu$, we obtain a bounded $2$-cocycle in
$Z^2_b(\Gamma,\Z(P)_{sa})$ by Lemma \ref{LogLem} (\ref{LogLem:Part1}) and the estimate $\|\psi\|\leq 2\sin^{-1}(3(8+6\sqrt{2})\beta/2)$
follows from the relation $|1-e^{it}|=2|\sin(t/2)|\leq |t|$. Note that
$3(4+3\sqrt{2})\beta<0.53$. For $0\leq t\leq 0.53$, the convexity of
$\sin^{-1}(t)$ yields $\sin^{-1}(t)\leq (\sin^{-1}(0.53)/0.53))t$, from which
 $\sin^{-1}(t)\leq 3t/(2\sqrt{2})$ follows by direct calculation. By hypothesis,
$\psi=\partial\phi$ for some $\phi\in C^1_b(\Gamma,\Z(P)_{sa})$ and, from Proposition \ref{K=4}, we may take
\begin{equation}
\|\phi\|\leq 6\|\psi\|\leq 12\sin^{-1}(3(4+3\sqrt{2})\beta)\leq
54(4+3\sqrt{2})\beta/\sqrt{2}=(162+108\sqrt{2})\beta.
\end{equation}
Lemma \ref{LogLem} (\ref{LogLem:Part2}) gives $\nu(g,h)=e^{i\partial\phi(g,h)}$. Then $|1-e^{it}|\leq |t|$ implies
 $
\|I_P-e^{i\phi(g)}\|\leq (162+108\sqrt{2})\beta$, $ g\in\Gamma$.
Defining $v_g'=e^{i\phi(g)}v_g$, we have, for $g\in \Gamma$,
\begin{equation}\label{eqestimate}
\|v'_g-u_g\|\leq\|v_g-u_g\|+\|e^{i\phi(g)}-I\|
\leq(170+114\sqrt{2})\beta.
\end{equation}
The unitaries $(v_g')_{g\in\Gamma}$ also satisfy
$v_g'x{v_g'}^*=\alpha_g(x)$ for $x\in P$ and since
$v_g'v_h'{v_{gh}'}^*=\omega(g,h)$, it follows  that $N$ is isomorphic to
$P\rtimes_{\alpha,\omega}\Gamma=M$. Further, Proposition
\ref{TCPProp} (\ref{TCPProp2}) gives an isomorphism
$\theta:M=P\rtimes_{\alpha,\omega}\Gamma\rightarrow N$ with $\theta(x)=x$ for
$x\in P$ and $\theta(u_g)=v_g'$ for $g\in\Gamma$.
Now $M$ and $N$ are both in standard position on $\Ks$ and $\beta<1/47$ so Lemma
\ref{generate} shows that $\xi$ is also a tracial vector for $N$ and $N'$. Thus we define a unitary $U\in\mathcal B(\Ks)$
 by $U(m\xi)=\theta(m)\xi$ for $m\in M$ and it is routine that
$\theta(m)=UmU^*$ for $m\in M$. As in the proof of Lemma \ref{generate} (iii), $\overline{u_gP\xi}=\overline{v_g'P\xi}$ for $g\in \Gamma$, so $U$ leaves these subspaces invariant. Then the result follows from \eqref{eqestimate} and the estimate
\begin{equation}
\|U-I_\Ks\|=\sup\,\{\|(U-I_\Ks )|_{\,\overline{u_gP\xi}\,}\|:g\in \Gamma\}
\leq\sup\,\{\|u_g-v_g'\|:g\in \Gamma\}.\qedhere
\end{equation}
\end{proof}

The reduction procedure of Section \ref{Cartan} can now be used to prove Part
(\ref{TB3}) of Theorem \ref{TB}.
\begin{theorem}\label{StrongKK}
Let $\alpha$ be a trace preserving, centrally ergodic
and properly outer action of a countable discrete group $\Gamma$ on a finite
amenable von Neumann algebra $P$ with  separable predual. Suppose that
$H^2_b(\Gamma,\Z(P)_{sa})=0$ and let  $M=P\rtimes_{\alpha,\omega}\Gamma$. Then, 
given a faithful unital normal representation $M\subseteq\mathcal B(\Hs)$
and another von Neumann algebra $N\subseteq\mathcal B(\Hs)$ with
$d(M,N)<\gamma< 5.77\times 10^{-16}$,  there is a $^*$-isomorphism
$\theta:M\rightarrow N$ with 
\begin{equation}
\|\theta(x)-x\|< 902\gamma+664\times 50948\times
(301\gamma)^{1/2},
\quad  x\in M,\ \|x\|\leq 1 .
\end{equation}
\end{theorem}
\begin{proof}
Take such a crossed product $M$. Suppose that $M\subseteq\mathcal B(\Hs)$ is a faithful
normal nondegenerate representation and suppose that $N\subseteq \mathcal B(\Hs)$
is another von Neumann algebra acting nondegenerately on $\Hs$ with $d(M,N)<\gamma<5.77\times 10^{-16}$. By Theorem
\ref{Injective} (\ref{Injective:Part1}), there is a unitary $u\in (M\cup N)''$ with $\|u-I_\Hs\|\leq
150\gamma$ and $uPu^*\subseteq N$. Set $N_1=u^*Nu$ so that $P\subseteq M\cap
N_1$. Then $d(M,N_1)\leq 301\gamma$. By Lemma \ref{Summation}, we can find a
Hilbert space $\Ks$ and faithful normal $^*$-representations
$\pi:M\rightarrow\mathcal B(\Ks)$ and $\rho:N_1\rightarrow\mathcal B(\Ks)$ so
that:
\begin{enumerate}[(i)]
\item\label{StrongKK:Item1} $\pi(M)$ and $\rho(N_1)$ are in standard position on $\Ks$ with common
tracial
vector $\xi$ for $\pi(M),\ \pi(M)',\ \rho(N_1)$, and $\rho(N_1)'$;
\item\label{StrongKK:Item2} $\pi(M)\subset_\beta\rho(N_1)$ and $\rho(N_1)\subset_\beta\pi(M)$ for some $\beta<50948\times (301\gamma)^{1/2}<1/47;$
\item\label{StrongKK:Item3} Given contractions $x\in M$ and $y\in N_1$, we have $\|\pi(x)-\rho(y)\|\leq \beta+\|x-y\|$;
\item\label{StrongKK:Item4} $\pi|_P=\rho|_P$;
\item\label{StrongKK:Item5} $\langle \pi(M),e_{\pi(P)}\rangle=\langle \rho(N_1),e_{\pi(P)}\rangle$, using $\xi$ for these basic constructions. 
\end{enumerate}
Condition (\ref{StrongKK:Item5}) ensures that, working with the modular conjugation operators induced
by $\xi$, we have
$J_{\pi(M)}\pi(P)J_{\pi(M)}=J_{\rho(N_1)}\pi(P)J_{\rho(N_1)}
\subseteq\rho(N_1)'$. Then
conditions (\ref{StrongKK:Item1}), (\ref{StrongKK:Item2}) and (\ref{StrongKK:Item4}) allow us to apply Lemma \ref{BoundedLem} to $\pi(M)$
and $\rho(N_1)$ on $\Ks$. Consequently there is a unitary $U\in \pi(P)'\cap
\langle
\pi(M),e_{\pi(P)}\rangle$ such that $U\pi(M)U^*=\rho(N_1)$ and $\|U-I_\Ks\|\leq
(170+114\sqrt{2})\beta$.  Define an isomorphism $\theta_1:M\rightarrow N_1$ by
$\rho^{-1}\circ\mathrm{Ad}(U)\circ \pi$. Given $x\in M$ with $\|x\|\leq 1$, fix
$y\in N_1$ with $\|x-y\|\leq301\gamma$. Then $\|\pi(x)-\rho(y)\|\leq301\gamma+\beta$
so that 
\begin{align}
\|\theta_1(x)-y\|&=\|U\pi(x)U^*-\rho(y)\|\leq
2\|U-I_\Ks\|+\|\pi(x)-\rho(y)\|\nonumber\\
&\leq(340+228\sqrt{2})\beta+301\gamma+\beta
<664\beta+301\gamma.
\end{align}
Let  $\theta=\mathrm{Ad}(u)\circ\theta_1$. Then, for $x\in M$ and $y\in N_1$ as above, we have
\begin{align}
\|\theta(x)-x\|&\leq 2\|u-I_{\Hs}\|+\|\theta_1(x)-y\|+\|y-x\|
<
902\gamma+664\beta\nonumber\\
&<902\gamma+664\times 50948\times (301\gamma)^{1/2}.\qedhere
\end{align}
\end{proof}

 The collection of groups all of whose actions satisfy the hypotheses of
Theorem \ref{StrongKK} contains $SL_n(\mathbb Z)$ for $n\geq 3$.
\begin{corollary}
For $n\geq 3$, let $\alpha:SL_n(\mathbb Z)\curvearrowright P$ be a centrally
ergodic, properly outer  trace preserving action on a finite amenable von
Neumann algebra $P$ with separable predual.  For each $\eps>0$, there exists
$\delta>0$ with the following property:  given a unital normal
representation $\iota:P\rtimes_\alpha\Gamma\rightarrow\mathcal B(\Hs)$ and a
II$_1$ factor $N\subseteq\mathcal B(\Hs)$ with
$d(\iota(P\rtimes_\alpha\Gamma),N)<\delta$,  there exists a surjective
$^*$-isomorphism $\theta:P\rtimes_\alpha\Gamma\rightarrow N$ with
$\|\theta-\iota\|<\eps$.
\end{corollary}

We now turn to examples of nonamenable II$_1$ factors which satisfy the
strongest form of the Kadison-Kastler conjecture and prove Theorem \ref{TA}.
Such factors must inevitably have the similarity property \cite{CCSSWW:InPrep}.
Due to the presence of property (T), we cannot construct crossed product factors
$P\rtimes_\alpha SL_n(\mathbb Z)$ for $n\geq 3$ with property $\Gamma$, so we tensor these crossed product factors with the hyperfinite II$_1$ factor to obtain the similarity property.

\begin{theorem}\label{SSKK}
Let $\alpha:\Gamma\curvearrowright P_0$ be a trace preserving, centrally ergodic
and properly outer action of a countable discrete group $\Gamma$ on a finite
amenable von Neumann algebra $P_0$ with separable predual and suppose that
$H^2_b(\Gamma,\Z(P_0)_{sa})=0$. Let
$M=(P_0\rtimes_\alpha\Gamma)\vnotimes R$, where $R$ denotes the hyperfinite
II$_1$ factor.  
 If
$M\subseteq\mathcal B(\Hs)$ is a unital normal representation of $M$ and
$N\subseteq\mathcal B(\Hs)$ is another von Neumann  algebra acting on $\Hs$
with $d(M,N)<\gamma< 10^{-9}<(2/5)\times (182722121)^{-1}$, 
then there exists a unitary $U\in \mathcal B(\Hs)$ with $UMU^*=N$ and 
$
\|U-I_\Hs\|\leq 646020405\gamma< 10^9\gamma.
$

In particular, $M$ is strongly Kadison-Kastler stable.
\end{theorem}

\begin{proof}
Write $M_0=P_0\rtimes_\alpha\Gamma$ so that $M=M_0\vnotimes R$. Now suppose
that $M$ is represented as a unital von Neumann subalgebra of $\mathcal B(\Hs)$
and $N$ is another von Neumann algebra on $\Hs$ with $d(M,N)<\gamma$. Then $\gamma<1/602$, so  Lemma
\ref{McDuffLem2} can be applied. Thus there exists a unitary $u\in (M\cup N)''$
with $\|u-I_\Hs\|\leq
150\gamma$ such that $uRu^*\subseteq N$, $P_0\subseteq R'\cap u^*Nu$ and $N$ is
generated by the subfactors $(uRu^*)'\cap N$ and $uRu^*$. In particular $N\cong ((uRu^*)'\cap
N)\vnotimes uRu^*$ by \cite{N:TMJ} and so $N$ is McDuff. Since $M$ and $N$ have
property $\Gamma$, Proposition \ref{DK} (\ref{DK:Part2}) gives the near inclusions
$M\subseteq_{\cb,5\gamma}N$ and $N\subseteq_{\cb,5\gamma}M$.

Write $N_1=u^*Nu$ and $P=(P_0\cup R)''$ so that $P\subseteq N_1\cap M$ and $N_1$
is generated by the commuting algebras $N_0=R'\cap N_1$ and $R$ on $\Hs$.  Since
$M\subseteq_{\cb,305\gamma}N_1$ and $N_1\subseteq_{\cb,305\gamma}M$, Lemma
\ref{RelCom} (\ref{RelCom:Part2}) gives $M_0=R'\cap M\subseteq_{\cb,305\gamma}R'\cap N_1$ and
$R'\cap N_1\subseteq_{\cb,305\gamma}M_0$.  Proposition \ref{cb-com} induces the
near inclusions  $M'\subseteq_{\cb,305\gamma}N_1'$ and
$N_1'\subseteq_{\cb,305\gamma}M'$.  By construction $P_0\subseteq M_0\cap
(R'\cap N_1)$.  

By Theorem \ref{Summation2}, applied with $\gamma_1=305\gamma$ replacing
$\gamma$ (valid as $305\gamma<1/(903\times 47)$), there exist representations
$\pi:M\rightarrow \mathcal B(\Ks)$ and $\rho:N_1\rightarrow\mathcal B(\Ks)$ which
agree on $P$  such that:
\begin{enumerate}[(i)]
\item there is a tracial vector $\xi\in\Ks$ for $\pi(M)$, $\pi(M)'$, $\rho(N_1)$
and $\rho(N_1)'$;
\item $J_{\pi(M)}\pi(P)J_{\pi(M)}\subseteq \rho(N_1)'$;
\item $\pi(M)\subseteq_\beta\rho(N_1)$ and $\rho(N_1)\subseteq_\beta\pi(M)$ for
$\beta=903\gamma_1=275415\gamma$;
\item given contractions $x\in M$ and $y\in N_1$, we have $\|\pi(x)-\rho(y)\|\leq\|x-y\|+903\gamma_1$.
\end{enumerate}
Uniqueness of standard representations (up to spatial isomorphism) allows us to assume that  $\Ks$ factorizes as
$\Ks_1\otimes \Ks_2$, where
$\Ks_1=\overline{\pi(M_0\otimes I_R)\xi}$ and
$\Ks_2=\overline{\pi(I_{M_0}\otimes R)\xi}$ and with the following additional
properties. The vector $\xi$ factorizes as
$\xi_1\otimes\xi_2$,  $\pi(M_0)$ acts in standard position on $\Ks_1$ with
respect to $\xi_1$ and $\pi(R)$ acts in standard position on $\Ks_2$ with
respect to $\xi_2$.   Consequently, with respect to this factorization,
$\pi(R)'\cap
(J_{\pi(M)}\pi(R)J_{\pi(M)})'=\mathcal B(\Ks_1)\otimes\mathbb CI_{\Ks_2}$. Since
$J_{\pi(M)}\pi(R)J_{\pi(M)}\subseteq \rho(N_1)'$, we have $\pi(R)'\cap
\rho(N_1)\subseteq \pi(R)'\cap
(J_{\pi(M)}RJ_{\pi(M)})'=\mathcal B(\Ks_1)\otimes\mathbb CI_{\Ks_2}$ and so the
factorization of $N_1=N_0\vnotimes R$ respects the decomposition of
$\Ks=\Ks_1\otimes\Ks_2$.

It follows that $\pi(M_0)$ and $\rho(N_0)$ can be regarded as represented on
$\Ks_1$ where $\xi_1$ is a tracial vector for
$\pi(M_0)$, $\pi(M_0)'$, $\rho(N_0)$, and $\rho(N_0)'$. Further,
$\pi(P_0)\subseteq \pi(M_0)\cap \rho(N_0)$ and
$J_{\pi(M_0)}\pi(P_0)J_{\pi(M_0)}\subseteq \rho(N_0)'$, where $J_{\pi(M_0)}$ is
the modular conjugation operator on $\Ks_1$ defined with respect to $\xi_1$.
Thus Lemma \ref{BoundedLem} gives a unitary $u_0\in\mathcal B(\Ks_1)$ with
$u_0\pi(M_0)u_0^*=\rho(N_0)$ and $\|u_0-I_{\Ks_1}\|\leq (170+114\sqrt{2})\beta$.
 Define $u_1=u_0\otimes I_{\Ks_2}$ so that $u_1$ is a unitary on $\Ks$ with
$u_1\pi(M)u_1^*=\rho(N_1)$ and $\|u_1-I_{\Ks}\|\leq (170+114\sqrt{2})\beta$.  

Define $\theta=\rho^{-1}\circ\Ad(u_1)\circ\pi:M\rightarrow N_1$. For a
contraction $x\in M$, choose a contraction $y\in N_1$ with $\|x-y\|\leq
301\gamma$ (possible as $d(M,N)<\gamma$ and $\|u-I_\Hs\|\leq150\gamma$).
Estimating in a very similar fashion to the end of the proof of Theorem
\ref{StrongKK} shows that
\begin{align}
\|\theta(x)-y\|&=\|u_1\pi(x)u_1^*-\rho(y)\|\leq
2\|u_1-I_\Hs\|+\|\pi(x)-\rho(y)\|\notag\\
&<(340+228\sqrt{2})\beta+301\gamma +\beta<182721820\gamma.
\end{align}
Thus
\begin{equation}
\|\theta(x)-x\|\leq \|\theta(x)-y\|+\|y-x\|\leq 182722121\gamma.
\end{equation}
By hypothesis this last quantity is less than $2/5$, so Lemma \ref{IsomorphismDK} (\ref{IsomorphismDK2}) applies. Therefore, there
exists a unitary $u_2$ on $\Hs$ with $\theta=\Ad(u_2)$ and 
\begin{equation}
\|u_2-I_\Hs\|\leq 2^{-1/2}\times 5\times 182722121\gamma\leq
646020255\gamma.
\end{equation}
Write $U=uu_2$ so that $UMU^*=N$. The proof is completed by the estimate
\begin{equation}
\|U-I_\Hs\|\leq \|u_2-I_\Hs\|+\|u-I_\Hs\|\leq 646020405\gamma<10^9\gamma.\qedhere
\end{equation}
\end{proof}

Theorem \ref{TA} now follows immediately from Theorem \ref{SSKK} and Theorem \ref{Monod}.

\begin{corollary}
Let $n\geq 3$ and $\Gamma=SL_n(\mathbb Z)$.  Given any free, ergodic, measure preserving action $\alpha:\Gamma\curvearrowright (X,\mu)$ on a standard probability space, the II$_1$ factor $(L^\infty(X,\mu)\rtimes_\alpha\Gamma)\vnotimes R$ is strongly Kadison-Kastler stable.
\end{corollary}

\begin{remark}\label{uncountable}
By combining Popa's superrigidity results for Bernouli actions $\Gamma\curvearrowright
(X,\mu)$ of ICC groups with property (T) (\cite[Theorem 0.1]{P:Invent3}) with Bowen's entropy invariant \cite{Bo:JAMS} for measure preserving actions of sofic groups and Popa's work on uniqueness of McDuff factorizations (\cite[Theorem 5.1]{P:IMRN}), there is a continuum of pairwise nonisomorphic II$_1$ factors of the form $(L^\infty(X)\rtimes_\alpha SL_3(\mathbb Z))\vnotimes R$ to which the previous result applies.
\end{remark}

Using the work in \cite{FM:TAMS1,FM:TAMS2}, it is also possible to prove Kadison-Kastler stability results for II$_1$ factors $M$ containing a Cartan masa $A$, which do not arise from the crossed product construction, when the associated equivalence relation satisfies cohomology conditions analogous to those in Theorems \ref{TA} and \ref{TB}.  Details can be found in the preprint version of this paper on the arXiv.

\subsection*{Acknowledgements} The authors gratefully acknowledge the following
additional sources of funding which enabled this research to be undertaken. A visit of EC
to Scotland in 2007 was supported by a grant from the Edinburgh Mathematical
Society; a visit of RS to Scotland in 2011 was supported by a grant from the
Royal Society of Edinburgh; SW visited Vassar college in 2010 supported by the
Rogol distinguished visitor program; JC, RS and SW visited Copenhagen in 2011
supported by the FNU of
Denmark.

SW would
like to thank Peter Kropholler and Nicolas Monod for helpful discussions about
group cohomology and Jesse Peterson for helpful discussions about group actions. The authors would like to thank Wai Kit Chan for his helpful
comments on earlier drafts of this paper and the anonymous referees for their useful suggestions and simplifications.

\providecommand{\bysame}{\leavevmode\hbox to3em{\hrulefill}\thinspace}

\providecommand{\href}[2]{#2}

\end{document}